\documentclass[12pt]{amsart}

\RequirePackage[colorlinks,citecolor=blue,urlcolor=blue]{hyperref}
\usepackage[T1]{fontenc}
\usepackage{lmodern}
\usepackage[english]{babel}
\usepackage{ upgreek }
\usepackage{amsmath}
\usepackage{amssymb}
\usepackage{amsthm}
\usepackage{bbm}
\usepackage{mathrsfs}
\usepackage{todonotes}

\newtheorem{theorem}{Theorem}[section]
\newtheorem{lemma}[theorem]{Lemma}

\newtheorem{corollary}[theorem]{Corollary}

\newtheorem{proposition}[theorem]{Proposition}

\newtheorem{condition}{Condition}

\newenvironment{conditionbis}[1]
  {\renewcommand{\thecondition}{\ref{#1}bis}
   \begin{condition}}
  {\end{condition}}

\numberwithin{equation}{section}

\newcommand*{\abs}[1]{\left\lvert#1\right\rvert}
\newcommand*{\norm}[1]{\left\lVert#1\right\rVert}
\newcommand*{\pent}[1]{\left[#1\right]}
\newcommand*{\sachant}[2]{\left.#1 \,\middle|\,#2\right.}
\def\bb#1{\mathbb{#1}}
\def\geq{\geqslant}
\def\leq{\leqslant}
\newcommand\ee{\varepsilon}
\DeclareMathOperator{\dd}{d\!}
\DeclareMathOperator{\e}{e}
\DeclareMathOperator{\sh}{sh}
\DeclareMathOperator{\Id}{Id}

\begin{document}
\title[Conditioned affine Markov walks]{Limit theorems for affine Markov walks conditioned to stay positive}
\author{Ion Grama}
\curraddr[Grama, I.]{ Universit\'{e} de Bretagne-Sud, LMBA UMR CNRS 6205,
Vannes, France}
\email{ion.grama@univ-ubs.fr}

\author{Ronan Lauvergnat}
\curraddr[Lauvergnat, R.]{Universit\'{e} de Bretagne-Sud, LMBA UMR CNRS 6205,
Vannes, France}
\email{ronan.lauvergnat@univ-ubs.fr}

\author{\'Emile Le Page}
\curraddr[Le Page, \'E.]{Universit\'{e} de Bretagne-Sud, LMBA UMR CNRS 6205,
Vannes, France}
\email{emile.le-page@univ-ubs.fr}

\date{\today}

\begin{abstract}
Consider the real Markov walk $S_n = X_1+ \dots+ X_n$ with increments $\left(X_n\right)_{n\geqslant 1}$ defined by a stochastic recursion starting at $X_0=x$. 
For a starting point $y>0$ denote by $\tau_y$ the exit time of the process $\left( y+S_n \right)_{n\geqslant 1}$ 
from the positive part of the real line.  
We investigate the asymptotic behaviour of the probability of the event $\tau_y \geqslant n$ 
and of the conditional law of $y+S_n$ given $\tau_y \geqslant n$ as $n \to +\infty$.
\end{abstract}

\maketitle

\section{Introduction}

Assume that the Markov chain $(X_n)_{n\geq 0}$ is defined by the stochastic recursion
\begin{equation}
	X_0=x \in \bb R, \qquad X_{n+1} = a_{n+1} X_n + b_{n+1}, \qquad n\geq 0,
\label{intr000}
\end{equation}
where $(a_i,b_i)_{i \geq 1}$ is a sequence of i.i.d.\ real random pairs satisfying $\bb E (\abs{a_1}^\alpha) = 1$ for some $\alpha >2$. 
Consider the Markov walk
$S_n=\sum_{k=1}^n X_k,$  $n\geq 1$.
Under a set of conditions ensuring the existence of the spectral gap of the transition operator of the Markov chain $\left( X_n \right)_{n\geq 0}$, it was established 
in Guivarc'h and Le Page \cite{guivarch_spectral_2008} that there exist constants $\mu$ and $\sigma>0$ such that, for any $t \in \bb R$,
\begin{equation}
\bb{P}_x \left(  \frac{S_n - n \mu }{\sigma\sqrt{n} }  \leq t    \right) \to\mathbf \Phi \left( t   \right) \quad  \text{as} \quad n\to +\infty,
\label{intr001}
\end{equation}
where $\mathbf \Phi $ is the standard normal distribution function and $\bb P_x $ is the probability measure generated by 
 $(X_n)_{n\geq 0}$ starting at $X_0=x$. There are easy expressions of $\mu$ and $\sigma$ in terms of law of the pair $(a,b)$: in particular $\mu = \frac{ \mathbb E b}{ 1-\mathbb E a}$.

For a starting point $y>0$, define the first time when the affine Markov walk $(y+S_n)_{n\geq 1}$ becomes 
non-positive by setting
	\[
	\tau_y = \min \{ k\geq 1,\,\, y+S_k \leq 0 \}.
\]
In this paper we complete upon the results in \cite{guivarch_spectral_2008} by determining the asymptotic of the probability $\bb P_x \left( \tau_y > n\right)$ and  proving a conditional version of the limit theorem \eqref{intr001} for the sum $y+S_n$, given the event $\{ \tau_y > n \}$ in the case when $\mu = 0.$ The main challenge in obtaining these asymptotics is to prove the existence of a positive harmonic function pertaining to the associated Markov chain $\left( X_n, y+S_n \right)_{n\geq 0}$. A positive harmonic function, say $V$, is defined as a positive solution of the equation $\textbf{Q}_+V=V$, where $\textbf{Q}_+$ is the restriction on $\bb R \times \bb R_+^*$ of the Markov transition kernel $\textbf{Q}$ of the chain $\left( X_n, y+S_n \right)_{n\geq 0}$.

From the more general results of the paper it follows that, under the same hypotheses that ensure the CLT 
(see Condition \ref{Mom001}   in Section \ref{sec Notations and results}), 
if the pair $(a,b)$ is such that $\bb P((a,b)\in (0,1) \times (0,C]  )>0$ and $\bb P((a,b)\in (-1,0)\times (0,C] )>0$, for some $C>0$, 
then
	\[
	\bb P_x \left( \tau_y > n \right)  \underset{n\to +\infty}{\sim} \frac{2V(x,y)}{\sqrt{2\pi n} \sigma}
\]
and
	\[
	\bb P_x \left( \sachant{ \frac{y+S_n}{\sigma \sqrt{n}} \leq t}{\tau_y > n} \right)  \underset{n\to +\infty}{\longrightarrow} \mathbf \Phi^+(t),
\]
where $\mathbf \Phi^+(t)= 1 - \e^{-t^2/2}$ is the Rayleigh distribution function.
In particular, the above mentioned results hold true if $a$ and $b$ are independent and $a$ is such that $\bb P(a\in (0,1))>0$ and $\bb P(a\in (-1,0))>0$. 
Less restrictive assumptions on the pair $(a,b)$  are formulated in our Section \ref{sec Notations and results}. 

The above mentioned results are in line with those already known in the literature for random walks with independent increments conditioned to stay in limited areas. We refer the reader to Iglehart \cite{iglehart_functional_1974}, Bolthausen \cite{bolthausen_functional_1976}, Doney \cite{doney_conditional_1985}, Bertoin and Doney \cite{bertoin_conditioning_1994}, Borovkov \cite{borovkov_asymptotic_2004-1,borovkov_asymptotic_2004}, Caravenna \cite{caravenna_local_2005}, Eichelsbacher and K\"{o}ning \cite{eichelsbacher_ordered_2008}, Garbit \cite{garbit_central_2009}, 
Denisov, Vatutin and Wachtel  \cite{denisov_local_2014}, Denisov and Wachtel \cite{denisov_conditional_2010,denisov_random_2015}. 
More general walks with increments forming a Markov chain have been considered by 
Presman \cite{presman_1967, presman_1969},
Varapoulos \cite{varopoulos_potential_1999,varopoulos_potential_2000}, Dembo \cite{dembo_persistence_2013}, 
Denisov and Wachtel \cite{denisov_exit_2015}  or Grama, Le Page and Peign\'e \cite{grama_conditional_2014}. 
In \cite{presman_1967, presman_1969} the case of sums of lattice random variables defined on finite regular Markov chains has been considered.
Varapoulos \cite{varopoulos_potential_1999, varopoulos_potential_2000} studied Markov chains with bounded increments and obtained lower and upper bounds for
the probabilities of the exit time from cones.   
Some  studies take advantage of additional properties: for instance     
in \cite{denisov_exit_2015} the Markov walk has a special integrated structure;
in \cite{grama_conditional_2014} the moments of $X_n$ are bounded by some constants not depending on the initial condition. However, to the best of our knowledge, the asymptotic behaviour of the probability $\bb P_x \left( \tau_y > n \right)$ 
in the case of the stochastic recursion (\ref{intr000}) has not yet been considered in the literature.

Note that the Wiener-Hopf factorization, which usually is employed in the case of independent random variables, 
cannot  be applied in a straightforward manner for  Markov chains.
Instead, to study the case of the stochastic recursion, we rely upon the developments in \cite{denisov_exit_2015},  \cite{denisov_random_2015} and \cite{grama_conditional_2014}. The main idea of the paper is given below. 
The existence of the positive harmonic function $V$ is linked to the construction of a martingale approximation 
for the Markov walk $\left(S_{n}\right)_{n\geq 1}$. While the harmonicity is inherently related to the martingale properties, the difficulty is to show that the approximating martingale is integrable at the exit time of the Markov walk $\left(y+S_{n}\right)_{n\geq 1}$.   
In contrast to \cite{denisov_random_2015} and \cite{grama_conditional_2014},
our proof of the existence of $V$ employs different techniques according to positivity or not of the values of $\bb E( a_1 )$. 
The constructed harmonic function allows to deduce the properties of the exit time and the conditional distribution of the Markov walk from those of the Brownian motion using a strong approximation result for Markov chains from Grama, Le Page and Peign\'{e} \cite{ion_grama_rate_2014}.  
The dependence of the constants on the initial state $X_0=x$ of the Markov chain $(X_n)_{n\geq 0}$ established there plays the essential role in our proof.

The technical steps of the proofs are as follows. We first deal with the case when the starting point of the Markov walk $(y+S_n)_{n\leq 0}$ is large: $y > n^{1/2-\ee},$ for some $\ee >0$.
When $y>0$ is arbitrary, the law of iterated logarithm ensures that the sequence $(\abs{y+S_k})_{1\leq k \leq n^{1-\ee}} $ will cross the level $n^{1/2-\ee}$ 
with high probability.
Then, by the Markov property, we are able to reduce the problem to a Markov walk with a large starting point $y'=y+S_{\nu_n}$, where 
$\nu_n$ is the first time  when the sequence $\abs{y+S_k}$ exceeds the level $n^{1/2-\ee}$.
The major difficulty, compared to \cite{denisov_random_2015} and \cite{grama_conditional_2014},  
is that, for the affine model under consideration,   
the sequence $\left(X_{\nu_n}\right)_{n\geq 1}$ is not bounded in $\bb L^1$.    
To overcome this we need a control of the moments of $X_n$ in function of the initial state $X_0=x$ and the lag $n.$   

We end this section by agreeing upon some basic notations. 
As from now and for the rest of this paper the symbols $c, c_{\alpha}, c_{\alpha,\beta},\dots$ denote positive constants depending only on their indices. All these constants are likely to change their values every \text{red}{occurrence}.
The indicator of an event $A$ is denoted by $\mathbbm 1_A.$
For any bounded measurable function $f$ on $\bb X=\bb R^d$,  $d = 1,2,$ 
random variable $X$ in $\bb X$ and event $A$, the integral $\int_{\bb X}  f(x) \bb P (X \in \dd x, A)$ means the expectation 
$\bb E\left( f(X); A\right)=\bb E \left(f(X) \mathbbm 1_A\right)$.  

\section{Notations and results}
\label{sec Notations and results}

Assume that on the probability space $(\Omega,\mathcal F,  \bb P)$ we are given a sequence of independent real random pairs $(a_i,b_i)$, $i\geq 1$,
of the same law as the generic random pair $(a,b)$. 
Denote by $\bb E$ the expectation pertaining to $\bb P$.
Consider the Markov chain $(X_n)_{n\geq 0}$  defined by the affine transformations
	\[
	X_{n+1} = a_{n+1} X_n + b_{n+1}, \qquad n\geq 0,
\]
where $X_0=x\in \bb R$ is a starting point. The partial sum process $(S_n)_{n\geq 0}$ defined by $S_n = \sum_{i=1}^{n} X_i$ for all $n\geq 1$ and $S_0=0$ will be called in the sequel
affine Markov walk. Note that $(S_n)_{n\geq 0}$ itself  is not a Markov chain, but the pair $(X_n,S_n)_{n\geq 0}$ forms a Markov chain.

For any $x \in \bb R$, denote by $\mathbf P (x,\cdot)$ the transition probability of $(X_n)_{n\geq 0}$. Introduce the transition operator
	\[
	\mathbf P f (x) = \int_{\bb R} f(x') \mathbf P (x,\dd x'),
\]
which is defined for any real bounded measurable function $f$ on $\bb R$. Denote by $\bb P_x$ and $\bb E_x$ the probability and the corresponding expectation generated by the finite dimensional distributions of $(X_n)_{n\geq 0}$ starting at $X_0=x$. Clearly, for any $x \in \bb R$ and $n\geq 1$,  we have 
$\mathbf{P}^n f \left( x\right) = \bb E_x \left(f\left( X_n \right)\right)$.

We make use of the 
following condition which ensures that the affine Markov walk satisfies the central limit theorem (\ref{intr001}) (c.f. \cite{guivarch_spectral_2008}): 
\begin{condition} The pair $(a,b)$ is such that:
\label{Mom001} 
\begin{enumerate}
\item There exists a constant $\alpha>2$ such that
$\bb E \left( \abs{a}^{\alpha} \right) < 1$ 
and 
$\bb E \left(\abs{b}^{\alpha} \right)  < +\infty.$
\item The random variable $b$ is non-zero with positive probability,  $\bb P (b\not=0) > 0$, and centred, $\bb E (b) = 0$.
\end{enumerate}
\end{condition}

Note that Condition \ref{Mom001} is weaker than the conditions required in \cite{guivarch_spectral_2008} in the special case $\alpha > 2$.
Nevertheless, using the same techniques as in \cite{guivarch_spectral_2008} it can be shown that, under Condition \ref{Mom001}, 
the Markov chain $(X_n)_{n\geq 0}$ has a unique invariant measure $\mathbf m$ and its 
partial sum $S_n$ satisfies the central limit theorem \eqref{intr001} with
\begin{equation}
  \label{defdemu}
\mu = \int_{\bb R} x \mathbf m (\dd x) = \frac{ \mathbb E (b)}{ 1-\mathbb E (a)} = 0 
\end{equation}
and
\begin{equation}
  \label{defdesigma}
	\sigma^2 = \int_{\bb R} x^2 \mathbf m (\dd x) + 2\sum_{k=1}^{\infty} \int_{\bb R} x \bb E_x (X_k) \mathbf m (\dd x)   
	= \frac{\bb E(b^2)}{1-\bb E(a^2)} \frac{1+\bb E(a)}{1-\bb E(a)} >0.
\end{equation}
Moreover, it is easy to see that under Condition \ref{Mom001} the Markov chain $(X_n)_{n\geq 0}$ has no fixed point: 
$\bb P \left( a x +b = x \right) < 1$, for any $x\in \bb R$.
Below we make use of a slightly refined result which gives the rate of convergence in the central limit theorem for $S_n$ with an explicit dependence of the constants on the initial value $X_0=x$ stated in Section \ref{Strong Approx}.

For any $y \in \bb R$ consider the affine Markov walk $\left(y+ S_n\right)_{n\geq 0}$ starting at $y$ and define its exit time
	\[
	\tau_y = \min \{ k \geq 1,\,\, y+S_k \leq 0 \}.
\]
Corollary \ref{Exitfinit} implies the finiteness of the stopping time $\tau_y$: under Condition \ref{Mom001}, it holds $\bb P_x \left(\tau_y <+\infty\right) = 1,$ for any $x\in \bb R$ and $y\in \bb R$.

The asymptotic behaviour of the probability $\bb P \left( \tau_y >n  \right) $ is determined by 
the harmonic function which we proceed to introduce.
For any $(x,y) \in \bb R \times \bb R$, denote by $\mathbf Q(x,y,\cdot )$ the transition probability of the Markov chain $(X_n,y+S_n)_{n\geq 0}.$  
The restriction of the measure $\mathbf Q(x,y,\cdot )$  on $\mathbb R\times \mathbb R^*_+$ is defined by
  \[
	\textbf{Q}_+(x,y,B) = \textbf{Q}(x,y,B)
\]
for any measurable set  $B$ on $\mathbb R\times \bb R_+^*$ and  for any $(x,y) \in \bb R \times \bb R$.  
Let $\mathscr D$ be a measurable set in $ \mathbb R\times \mathbb R$ containing $\mathbb R\times \bb R_+^*$.
For any measurable $\varphi: \mathscr D \rightarrow \bb R$ set $\textbf{Q}_+\varphi (x,y)=\int_{\mathbb R\times \bb R_+^*} \varphi(x',y') \textbf{Q}_+(x,y,\dd x' \times \dd y').$
A positive $\textbf{Q}_+$-harmonic function on $\mathscr D$ is any function $V: \mathscr D \rightarrow \bb R$ which satisfies
	\[
	\textbf{Q}_+V (x,y)  = V(x,y)>0, \quad \text{for any}\quad (x,y) \in \mathscr D.
\]

To ensure the existence of a positive harmonic function we need additional assumptions: 
\begin{condition}
\label{PosdeX1-1}  For all $x \in \bb R$ and $y>0$,
	\[
	 \bb P_x \left( \tau_y > 1 \right) = \bb P \left( a x + b > -y \right) > 0.
\]
\end{condition}

\begin{condition}
\label{PosdeX1-2}
For any  $x \in \bb R$ and $y>0$, there exists $p_0 \in (2,\alpha)$ such that for any constant $c>0$, there exists $n_0 \geq 1$ such that,
	 \[
	\bb P_x \left( \left( X_{n_0}, y+S_{n_0} \right) \in K_{p_0,c} \,,\, \tau_y > n_0 \right) > 0,
 \]
where
	\[
	K_{p_0, c} = \left\{ (x,y) \in \bb R \times \bb R_+^*, y \geq c \left( 1+ \abs{x}^{p_0} \right) \right\}.
\]
\end{condition}

Obviously Condition \ref{PosdeX1-1} is equivalent to $\bb P_x \left( X_1 > -y \right) = \bb P_x \left( \tau_y > 1 \right) > 0$ for any $x\in \bb R$ and $y>0$, which, in turn is equivalent to the fact that there exists $n_0 \geq 1$ such that $\bb P_x \left( \tau_y > n_0 \right) > 0$, for any $x \in \bb R$ and $y>0$. Therefore Condition \ref{PosdeX1-2} implies Condition \ref{PosdeX1-1}.
As a by-product, under either Condition \ref{PosdeX1-1} or Condition \ref{PosdeX1-2}, the event $\{ \tau_y >n \}$ is not empty. 

The existence of a harmonic function is guaranteed by the following theorem. For any $x \in \bb R$ 
consider the $\bb P_x$-martingale $(M_n, \mathscr{F}_n)_{n\geq 0}$ defined by
\begin{equation}
\label{mart001}
	M_n = S_n + \frac{\bb E(a)}{1-\bb E(a)} \left( X_n -x \right), \quad n\geq 0,
\end{equation}
with $(\mathscr{F}_n)_{n\geq 0}$ the natural filtration (we refer to Section \ref{Mart Approx} for details).

\begin{theorem}
\label{ExofHaFu}
Assume either Conditions \ref{Mom001}, \ref{PosdeX1-1} and  $\bb E(a) \geq 0$, or Conditions \ref{Mom001} and \ref{PosdeX1-2}.
\begin{enumerate}
\item For any $x \in \bb R$ and $y >0$, the random variable $M_{\tau_y}$ is integrable,
	\[
	\bb E_x \left( \abs{M_{\tau_y}} \right) < +\infty
\]
and the function
	\[
	V(x,y) = -\bb E_x \left( M_{\tau_y} \right), \quad x\in \bb R, \ y>0,
\]
is well defined on $\bb R \times \bb R_+^*$.
\item The function $V$ is positive and  $\textbf{Q}_+$-harmonic on $\bb R \times \bb R_+^*$: for any $x \in \bb R$ and $y >0$, 
	\[
	\textbf{Q}_+V(x,y) = V(x,y). 
\]
\item Moreover, the function $V$ has the following properties:
\begin{enumerate}
\item For any $x \in \bb R$, the function $V(x,.)$ is non-decreasing.
\item For any $\delta >0$, $p \in (2,\alpha)$,  $x\in \bb R$ and $y>0$,
\begin{align*}
V(x,y) &\geq \max\left(0,(1-\delta) y -  c_{p,\delta} \left(1+\abs{x}^{p}\right)  \right), \\
V(x,y) &\leq  \left( 1+ \delta \left(1+ \abs{x}^{p-1} \right)\right) y + c_{p,\delta}  \left(1+ \abs{x}^{p} \right).
\end{align*}
\item For any $x \in \bb R$, it holds $\underset{y\to +\infty}{\lim} \frac{V(x,y)}{y} = 1.$
\end{enumerate}
\end{enumerate}
\end{theorem}

Using the harmonic function from the previous theorem, we obtain the asymptotic of the tail probability of the exit time $\tau_y$. 
\begin{theorem}
\label{AsExTi}
Assume either Conditions \ref{Mom001}, \ref{PosdeX1-1} and  $\bb E(a) \geq 0$, or Conditions \ref{Mom001} and \ref{PosdeX1-2}.
\begin{enumerate}
\item For any $p\in (2,\alpha)$, $x \in \bb R$ and $y>0$,
\[ \sqrt{n}\bb P_x \left( \tau_y > n \right) \leq c_p \left( 1+ y + \abs{x} \right)^p. \]
\item For any $x \in \bb R$ and $y>0$,
\[ \bb P_x \left( \tau_y > n \right)  \underset{n\to +\infty}{\sim} \frac{2V(x,y)}{\sqrt{2\pi n} \sigma}. \]
\end{enumerate}
\end{theorem}

\begin{corollary}
\label{MoofExTi}
Assume either Conditions \ref{Mom001}, \ref{PosdeX1-1} and  $\bb E(a) \geq 0$, or Conditions \ref{Mom001} and \ref{PosdeX1-2}. For any $p \in (2,\alpha)$, $x \in \bb R$, $y>0$ 
and $\gamma \in (0,1/2)$,
	\[
	\bb E_x \left( \tau_y^{\gamma} \right) \leq  c_{p,\gamma} (1+ y + \abs{x})^p       \qquad \text{and} \qquad \bb E_x \left( \tau_y^{1/2} \right) = +\infty.
\]
\end{corollary}

Moreover, we prove that the Markov walk $\left(y+S_n\right)_{n\geq 0}$ conditioned to stay positive satisfies the following limit theorem.

\begin{theorem}
\label{AsCoRaWa} 
Assume either Conditions \ref{Mom001}, \ref{PosdeX1-1} and  $\bb E(a) \geq 0$, or Conditions \ref{Mom001} and \ref{PosdeX1-2}.
For any $x \in \bb R$, $y>0$ and $t>0$,
	\[
	\bb P_x \left( \sachant{ \frac{y+S_n}{\sigma \sqrt{n}} \leq t}{\tau_y > n} \right)  \underset{n\to +\infty}{\longrightarrow} \mathbf \Phi^+(t),
\]
where $\mathbf \Phi^+(t) = 1-\e^{-\frac{t^2}{2}}$ is the Rayleigh distribution function.
\end{theorem}

Theorems  \ref{ExofHaFu},   \ref{AsExTi},  \ref{AsCoRaWa}  can be extended to some non-positive initial points $y$. Set
\[ \mathscr{D}^- := \left\{ (x,y) \in \bb R\times \bb R_-, \; \bb P_x \left( \tau_y > 1 \right) = \bb P \left( ax+b>-y \right) >0 \right\}. \]

\begin{theorem}
\label{initialpointneg}
Assume either Conditions \ref{Mom001}, \ref{PosdeX1-1} and  $\bb E(a) \geq 0$, or Conditions \ref{Mom001} and \ref{PosdeX1-2}.
\begin{enumerate}
\item For any $(x,y) \in \mathscr{D}^-$, the random variable $M_{\tau_y}$ is integrable and the function $V(x,y) = -\bb E_x \left( M_{\tau_y} \right)$, is well defined on $\mathscr{D}^-$.
\item The function $V$ is positive and  $\textbf{Q}_+$-harmonic on $\mathscr{D}= \mathscr{D}^- \cup \bb R \times \bb R_+^*$.
\item \begin{enumerate}
\item For any $(x,y) \in \mathscr{D}^-$,
\[ \sqrt{n}\bb P_x \left( \tau_y > n \right) \leq c_p \left( 1+ \abs{x} \right)^{p}. \]
\item For any $(x,y) \in \mathscr{D}^-$,
\[ \bb P_x \left( \tau_y > n \right)  \underset{n\to +\infty}{\sim} \frac{2V(x,y)}{\sqrt{2\pi n} \sigma}. \]
\end{enumerate} 
\item For any $(x,y) \in \mathscr{D}^-$ and $t>0$,
	\[
	\bb P_x \left( \sachant{ \frac{y+S_n}{\sigma \sqrt{n}} \leq t}{\tau_y > n} \right)  \underset{n\to +\infty}{\longrightarrow} 
	\mathbf \Phi^+(t).
\]
\end{enumerate}
\end{theorem}

Below we discuss two more restrictive assumptions which, however, are easier to verify than Conditions \ref{PosdeX1-1} and \ref{PosdeX1-2}, respectively.

\begin{conditionbis}{PosdeX1-1}
\label{CSPosdeX1-1}
The law of the pair $(a,b)$ is such that for all $C>0$,
	\[
	\bb P \left( b \geq C \abs{a} \right) > 0.
\]
\end{conditionbis}

\begin{conditionbis}{PosdeX1-2}
\label{CSPosdeX1-2}
There exists $C>0$ such that,
	\[
	\bb P \left( (a,b) \in (-1,0) \times (0,C] \right) > 0 \qquad \text{and} \qquad \bb P \left( (a,b) \in (0,1) \times (0,C] \right) > 0.
\]
\end{conditionbis}

It is straightforward that Condition \ref{CSPosdeX1-1} implies Condition \ref{PosdeX1-1}. This follows from the inequality
	\[
	\bb P \left( a x + b > -y \right) \geq \bb P \left( b \geq  C \abs{a} \right),
\]
with $C=\abs{x}$. The fact that Condition \ref{CSPosdeX1-2} implies Condition \ref{PosdeX1-2} is proved in the Appendix \ref{ComplementCond}.

Under Condition \ref{Mom001}, it is easy to see that Condition \ref{CSPosdeX1-2} is satisfied, for example, when random variables $a$ and $b$ are independent and $P \left( a \in (-1,0) \right) > 0$ and $P \left( a \in (0,1) \right) > 0$.

Note that, while Condition \ref{PosdeX1-2} implies Condition \ref{PosdeX1-1}, there is no link between Conditions \ref{CSPosdeX1-1} and \ref{CSPosdeX1-2}. Indeed, if $a$ and $b$ are independent, $a$ is non-negative and the support of $b$ contains $\bb R_+$, then Condition \ref{CSPosdeX1-1} holds true whereas Condition \ref{CSPosdeX1-2} does not. At the opposite, if $a$ and $b$ are independent $b$ bounded and support of $a$ equal to $\{ -1/2 \} \cup \{1/2\}$ then Condition \ref{CSPosdeX1-2} 
holds true whereas Condition \ref{CSPosdeX1-1} does not.

The outline of the paper is as follows. The martingale approximation $\left( M_n \right)_{n\geq 0}$ of the Markov walk $\left( S_n \right)_{n\geq 0}$ and some of its properties are given in Section \ref{Mart Approx}. In Section \ref{CMWI} we prove that the expectation of the killed Markov walk $( \left(y+S_n\right) \mathbbm 1_{\{ \tau_y >n \}} )_{n\geq 0}$ is bounded uniformly in $n$. This allows us to prove the existence of the harmonic function and establish some of its properties in Section \ref{Sec Harm Func}. With the help of the harmonic function and of a strong approximation result for Markov chains we prove Theorems \ref{AsExTi}, \ref{AsCoRaWa} and \ref{initialpointneg}, in Sections \ref{As for Exit Time}, \ref{As for cond Markov walk} and \ref{secproofofinitialpointneg} respectively. Section \ref{Appendix} is an appendix where we collect some results used in the proofs.

\section{Martingale approximation}
\label{Mart Approx}

In this section we approximate the Markov walk $\left(S_n \right)_{n\geq 0}$ by a martingale following Gordin \cite{gordin_central_1969}. 
We precede this construction by a lemma which shows that there is an exponential decay of the dependence of $X_n$ on the initial state $x=X_0$ as $n$ grows to infinity. This simple fact will be used repeatedly in the sequel.

\begin{lemma}
\label{MCM}
For all $p \in [1,\alpha]$, $x\in \bb R$, and $n\geq 0$,
	\[
	\bb E_x^{1/p} \left( \abs{X_n}^p \right) \leq c_p + \left( \bb E^{1/p} \left( \abs{a}^p \right) \right)^n \abs{x}  \leq c_p ( 1 + \abs{x} ).
\]
\end{lemma}

\begin{proof}
Since $X_n = \sum_{k=1}^{n} \left(b_k \prod_{i=k+1}^n a_i\right) + \prod_{i=1}^n a_i x$, for $n\geq 1$, with the convention $\prod_{i=n+1}^n a_i =1$, we have by the Minkowski inequality and the independence of $(a_i,b_i)_{i\geq 1}$, 
\[
	\bb E_x^{1/p} \left( \abs{X_n}^p \right) \leq \sum_{k=1}^{n} \left(\bb E^{1/p} \left( \abs{b}^p \right) \bb E^{1/p} \left( \abs{a}^p \right)^{n-k}\right) + \bb E^{1/p} \left( \abs{a}^p \right)^n \abs{x}. 
\]
The conclusion of the lemma is thus a direct consequence of Condition \ref{Mom001}.
\end{proof}

Let $\Id(x) =x$, $x\in \bb R$ be the identity function on $\bb R$. The Poisson equation $u - \mathbf P u = \Id$ has a unique solution $\theta$, given by,
	\[
	\theta(x)=  \sum_{k=0}^{+\infty} \mathbf P^k \Id(x) = x + \sum_{k=1}^{+\infty} \mathbb E_x \left( X_k \right)	= x + \sum_{k=1}^{+\infty} \bb E(a)^k x = \frac{x}{1-\bb E(a)}.
\]
Using the function $\theta$, the process $(M_n)_{n\geq 0}$ defined in (\ref{mart001}) can be recast as
	\[
	M_0=0, \qquad M_n = \sum_{k=1}^{n} \theta \left(X_k\right) - \mathbf P \theta \left(X_{k-1}\right) = \sum_{k=1}^{n} \frac{X_k - \bb E(a) X_{k-1}}{1-\bb E(a)}, \quad n\geq 1.
\]
Consider the natural filtration $\left(\mathscr{F}_n\right)_{n \geq 0}$ with $\mathscr{F}_0$ the trivial $\sigma$-algebra and 
$\mathscr{F}_n$ the $\sigma$-algebra generated by $X_1,\,X_2,\dots,\,X_n.$
The fact that $(M_n,\mathscr{F}_n)_{n\geq 0}$ is indeed a $\bb P_x$-martingale, for any $x \in \bb R$,
is easily verified by the Markov property: $\bb E_x (\theta(X_{n+1})| \mathscr{F}_{n}) = \mathbf P \theta \left(X_{n}\right),$ for $n\geq 0.$

All over the paper we use the abbreviation 
\begin{equation}
	\rho= \frac{\bb E(a)}{1-\bb E(a)}.
\label{RHO-001}
\end{equation}
With this notation, for any $x \in \bb R$ and $y \in \bb R$, the Markov walk
$\left( y + S_n \right)_{n\geq 0}$ has the following martingale representation:
\begin{equation}
	\label{MfSetX}
	y + S_n = y + \rho x+M_n -\rho X_n, \quad n\geq 0.
\end{equation}

Define the sequence $(X_n^0)_{n \geq 0}$, by
\begin{equation}
	\label{defdeXn0}
	X_0^0=0 \qquad \text{and} \qquad X_n^0 = \sum^{n}_{k=1} b_k \prod^{n}_{i=k+1} a_i, \quad n \geq 1,
\end{equation}
with the convention $\prod^{n}_{i=k+1} a_i=1$ for $k=n$. The sequence $(X_n^0)_{n \geq 0}$ corresponds to the stochastic recursion starting at $0$. In the same line, we define $M_0^0=0$ and $M_n^0 = \sum_{k=1}^{n} \frac{X_k^0 - \bb E(a) X_{k-1}^0}{1-\bb E(a)}$, for all $n\geq 1$. It is easy to see that the process $\left( M_n^0,\mathscr{F}_n \right)_{n\geq 0}$ is a zero mean $\bb P_x$-martingale which is related to the martingale $\left( M_n\right)_{n\geq 0}$ by the identity
\begin{equation}
 \label{MDecenx} M_n =  M_n^0 + \Delta_n  x,
\end{equation}
where
	\[
\Delta_0=0 \qquad \text{and} \qquad 	
\Delta_n= \sum_{k=1}^{n} \frac{\prod_{i=1}^{k-1} a_i}{1-\bb E(a)} \left(a_k- \bb E(a)\right), \quad n \geq 1. 
\]

The following two lemmas will be used to control $\bb E_x ( \abs{M_n}^p )$.

\begin{lemma}
\ 
\label{aboutD}
\begin{enumerate}
\item The sequence $(\Delta_n)_{n\geq 0}$ is a centred martingale.
\item For all $p \in [1,\alpha)$ and $n \geq 0$,
	\[
	\bb E^{1/p} \left( \abs{\Delta_n}^p \right) \leq c_p.
\]
\end{enumerate}
\end{lemma}

\begin{proof}
The first claim follows from the fact that $\Delta_n$ is a difference of two martingales.
Using the Minkowski inequality for $1 \leq p < \alpha$, the independence of $(a_i)_{i\geq 1}$ and Condition \ref{Mom001} we obtain the second claim.
\end{proof}

Let us introduce the martingale differences: 
	\[
	\xi_k^0 = M_k^0 - M_{k-1}^0 = \frac{X_k^0 - \bb E(a) X_{k-1}^0}{1- \bb E(a)}, \quad k \geq 1.
\]

\begin{lemma}
\label{majmart0}
For all $p \in [1,\alpha)$ and $n \geq 0$,
	\[
	\bb E^{1/p} \left( \abs{\xi_n^0}^p \right) \leq c_p \qquad \text{and} \qquad \bb E^{1/p} \left( \abs{M_n^0}^p \right) \leq c_p \sqrt{n}.
\]
\end{lemma}

\begin{proof}
For the increments $\xi_n^0$ we simply use Lemma \ref{MCM} with $x=0$. For the martingale $(M_n^0)_{n\geq 0}$, the upper bound is obtained by Burkholder inequality: for all $2 < p < \alpha$ and all $n\geq 1$,
	\[
	\bb E^{1/p} \left( \abs{M_n^0}^p \right) \leq c_p \bb E^{1/p} \left( \left( \sum_{k=1}^n \left(\xi_k^0\right)^2 \right)^{p/2} \right).
\]
By the H\"older inequality with the exponents $u=p/2>1$ and $v=\frac{p}{p-2}$, we obtain
	\[
	\bb E^{1/p} \left( \abs{M_n^0}^p \right) \leq c_p \bb E^{1/p} \left[ \left( \sum_{k=1}^n \left(\xi_k^0\right)^{2u} \right)^{\frac{p}{2u}} n^{\frac{p}{2v}} \right] \leq c_p n^{\frac{p-2}{2p}} \left( \sum_{k=1}^n c_p \right)^{1/p} = c_p \sqrt{n}.
\]
This proves the claim when $2 < p < \alpha$. When $1 \leq p \leq 2$ the assertion follows obviously using Jensen inequality.
\end{proof}

\begin{lemma}
\label{majmartx}
For all $p \in [1,\alpha)$ and $n \geq 0$,
	\[
	\bb E_x^{1/p} \left( \abs{M_n}^p \right) \leq c_p \left( \abs{x} + \sqrt{n} \right).
\]
\end{lemma}

\begin{proof}
By the Minkowski inequality and equation (\ref{MDecenx}), for all $1\leq p < \alpha$, $x\in \bb R$ and $n\geq 1$,
	\[
	\bb E^{1/p}_x \left( \abs{M_n}^p \right) \leq \bb E^{1/p} \left( \abs{\Delta_n}^p \right) \abs{x} + \bb E^{1/p} \left( \abs{M_n^0}^p \right).
\]
Then, by the claim 2 of Lemma \ref{aboutD} and Lemma \ref{majmart0}, the result follows.
\end{proof}

\section{Integrability of the killed martingale}
\label{CMWI}

The goal of this section is to prepare the background to prove the integrability of the random variable 
$ M_{\tau_y},$ which is crucial for showing the existence of the harmonic function in Section \ref{Sec Harm Func}. We use different approaches depending on the sign on $\bb E(a)$: 
when $\bb E(a) \geq 0$, in Section \ref{case >0} we prove that the expectation of the martingale 
$(y+\rho x + M_n )_{n\geq 0}$ killed at $\tau_y$ is uniformly bounded in $n$,
while, when $\bb E(a) < 0$, in Section \ref{CMWIcasneg} we prove that the expectation of  the same martingale 
killed at $T_y$ is uniformly bounded in $n$,
where $T_y$ is the exit time of the martingale $\left( y+\rho x+M_n \right)_{n\geq 0}$.

\subsection{Preliminary results} \label{prelim rez}
We first state a result concerning the first time when the process $\left( \abs{y+S_n} \right)_{n\geq 1}$ 
(respectively $\left( \abs{y+\rho x+M_n} \right)_{n\geq 1}$) crosses the level $n^{1/2-2\ee}$. Introduce the following stopping times: for any $n\geq 1$, $\ee \in (0,1/2)$, $x\in \bb R$ and $y \in \bb R$,
	\[
	\nu_n = \nu_{n,\ee,y} = \min \left\{ k \geq 1,\,\, \abs{y+S_k} > n^{1/2-\ee} \right\}
\]
and
	\[ v_n = v_{n,\ee,x,y} = \min \left\{ k \geq 1,\,\, \abs{y+\rho x+M_k} > n^{1/2-\ee} \right\}.
\]

\begin{lemma}
\label{concentnu}
Let $p \in (2,\alpha)$. There exists $\ee_0>0$ such that for any $\ee \in (0,\ee_0]$, $\delta>0$, $x \in \bb R$, $y>0$ and $n \geq 1$,
\[
	\bb P_x \left( \nu_n > \delta n^{1-\ee}  \right) \leq \frac{c_{p,\ee,\delta}}{ n^{p/2-p\ee} } + \e^{-c_{p,\ee,\delta} n^{1-2\ee}} \abs{x}^p
\]
and
\[
	\bb P_x \left( v_n > \delta n^{1-\ee} \right) \leq \frac{c_{p,\ee,\delta}}{ n^{p/2-p\ee} } + \e^{-c_{p,\ee,\delta} n^{1-2\ee}} \abs{x}^p.
\]
\end{lemma}

\begin{proof}
With $\ee <\min(1/2,\ee_0)$, where $\ee_0$ is defined in Corollary \ref{BerEss} and  $b>0$ a constant to be chosen below, let $l=\pent{b^2 \delta  n^{1-2\ee}}$, $K=\pent{n^\ee/b^2}$ and for any $m\geq 1$, $x\in \bb R$ and $y \in \bb R$, with $z=y+\rho x$,
\[
A_m(x,y)= \left\{ \max_{1\leq k \leq m} \abs{z+M_{kl}} \leq (1+2\abs{\rho}) n^{1/2-\ee} \right\}.
\]
Note that by the martingale representation (\ref{MfSetX}), we have for any $k \geq 2$, $\abs{z+M_k}$ $= \abs{y+S_k + \rho (y+S_k) - \rho (y+S_{k-1})} \leq (1+\abs{\rho}) \abs{y+S_k} + \abs{\rho} \abs{y+S_{k-1}}$. Then, choosing $n$ large enough to have $l \geq 2$,
\begin{align*}
	\bb P_x \left( \nu_n > \delta n^{1-\ee} \right) &= \bb P_x \left( \underset{1\leq k \leq \pent{\delta n^{1-\ee}}}{\max} \abs{y+S_k} \leq n^{1/2-\ee} \right) \\
	&\leq \bb P_x \left( \underset{2\leq k \leq \pent{\delta n^{1-\ee}}}{\max} \abs{z+M_k} \leq (1+2\abs{\rho}) n^{1/2-\ee} \right) \\
	&\leq \bb P_x \left( A_K(x,y) \right).
\end{align*}
Moreover, we have also,
\[
\bb P_x \left( v_n > \delta n^{1-\ee} \right) \leq \bb P_x \left( A_K(x,y) \right).
\]
Since $(X_n,y+S_n)_{n\geq 0}$ is a Markov chain,
\begin{align}
	\bb P_x \left( A_K(x,y) \right) &= \int_{\bb R^2} \bb P_{x'} \left( A_1(x',y') \right) \nonumber\\
	&\hspace{2cm} \times \bb P_{x} \left( X_{(K-1)l} \in \dd x' \,,\, y+S_{(K-1)l} \in \dd y' \,,\, A_{K-1}(x,y) \right).
	\label{pK001}
\end{align}
We use the decomposition (\ref{MDecenx}) to write that, with $c=1+2\abs{\rho}$,
\begin{align*}
\bb P_{x'}	\left(A_1(x',y')\right) \leq& \; \bb P_{x'} \left( \abs{z'+M_l^0} \leq 2cn^{1/2-\ee} \,,\, \abs{\Delta_l x'} \leq cn^{1/2-\ee} \right) \\
&+ \bb P_{x'} \left( \abs{\Delta_l x'} > cn^{1/2-\ee}\right).
\end{align*}
Using (\ref{MfSetX}) with $x=0$, we have $M_n^0=S_n^0+\rho X_n^0$. By the Markov inequality,
\begin{align*}
\bb P_{x'}  \left( A_1(x',y')\right) \leq & \; \bb P_{x'} \left( \abs{z'+S_l^0} \leq 3cn^{1/2-\ee} \,,\, \abs{\rho} \abs{X_l^0} \leq cn^{1/2-\ee} \right)  \\
   & + \bb P_{x'} \left( \abs{\rho} \abs{X_l^0} > cn^{1/2-\ee}\right) + c_p\frac{\bb E \left( \abs{\Delta_l}^p \right)}{ n^{p/2-p\ee} } \abs{x'}^p.
\end{align*}
Since $S^0_l$ does not depend on $x'$, using Lemma \ref{MCM} and the claim 2 of Lemma \ref{aboutD}, we obtain
\[
\bb P_{x'} \left(A_1(x',y')\right) \leq \underset{y'\in \bb R}{\sup}\, \bb P \left( \abs{y'+S_l^0} \leq 3cn^{1/2-\ee} \right) + \frac{c_p \left( 1+\abs {x'}^p \right)}{ n^{p/2-p\ee} }.
\]
Inserting this bound in (\ref{pK001}), it follows that
\begin{align*}
\bb P_x \left( A_K(x,y) \right) \leq&\; \bb P_x \left( A_{K-1}(x,y) \right) \underset{y'\in \bb R}{\sup}\, \bb P \left( \abs{y'+S_l^0} \leq 3cn^{1/2-\ee} \right) \\
&+ \frac{c_p}{ n^{p/2-p\ee} } \left( 1+ \bb E_{x} \left( \abs{X_{(K-1)l}}^p \right) \right).
\end{align*}
Set $r_n=\frac{3cn^{1/2-\ee}}{\sqrt{l}}$. Denote by $\bb B_{\frac{ -y'}{\sqrt{l}}} (r_n)$ the closed ball centred 
in $\frac{-y'}{\sqrt{l}}$ of radius $r_n$. The rate of convergence in the central limit theorem from Corollary \ref{BerEss} (applied with $x=0$) implies that,
	\[
	\underset{y'\in \bb R}{\sup}\, \bb P \left( \frac{S_l^0}{\sqrt{l}} \in \bb B_{\frac{-y'}{\sqrt{l}}} (r_n) \right) \leq 
	\underset{y'\in \bb R}{\sup}\, \int_{\bb B_{\frac{-y'}{\sqrt{l}}} (r_n)} \e^{-\frac{u^2}{2\sigma^2}} \frac{\dd u}{\sqrt{2\pi} \sigma} 
	+ 2\frac{c_{p,\ee}}{l^\ee}.
\]
Moreover,
	\[
	\underset{y'\in \bb R}{\sup}\, \int_{\bb B_{\frac{-y'}{\sqrt{l}}} (r_n)} \e^{-\frac{u^2}{2\sigma^2}} \frac{\dd u}{\sqrt{2\pi} \sigma}  
		  \leq \frac{2r_n}{\sqrt{2\pi}\sigma}
		\leq \frac{c_\delta}{b}.
\]
Let $q<1$.  With $b$ large enough in the definition of $l$, we have $2\frac{c_{p,\ee}}{l^\ee} \leq \frac{q}{2}$, $\frac{c_\delta}{b} \leq \frac{q}{2}$ and thus
\[
\underset{y'\in \bb R}{\sup}\, \bb P \left( \frac{S_l^0}{\sqrt{l}} \in \bb B_{\frac{-y'}{\sqrt{l}}} (r_n) \right) \leq q < 1.
\]
Iterating, we get
	\[
	\bb P_x \left( A_K(x,y) \right) \leq q^{K-1} \bb P_x \left( A_1(x,y) \right) + \frac{c_p}{ n^{p/2-p\ee} } \sum_{k=0}^{K-2} q^k \left( 1+ \bb E_{x} \left( \abs{X_{(K-1-k)l}}^p \right) \right).
\]
Using 
the fact that
$q^{K-1} \bb P_x \left( A_1 (x,y) \right) \leq q^{K-1} = \frac{1}{q} \e^{-\pent{n^\ee/b^2} \ln(1/q)} \leq \frac{c_{p,\ee,\delta}}{ n^{p/2-p\ee} }$, 
Lemma \ref{MCM} 
and the fact that $(K-1-k)l \geq c_{\ee,\delta} n^{1-2\ee}$ for all $0 \leq k \leq K-2$,
we finally obtain
	\[
	\bb P_x \left( A_K(x,y) \right) \leq \frac{c_{p,\ee,\delta}}{ n^{p/2-p\ee} } + \e^{-c_{p,\ee,\delta} n^{1-2\ee}} \abs{x}^p.
\]

\end{proof}

\subsection{Integrability of the killed martingale: the case $\bb E(a) \geq 0$}
\label{case >0}
The difficulty in proving that the expectation $\bb E_x  ( y+\rho x + M_n\, ; \, \tau_y > n )$ is integrable lies in the fact that whereas the killed Markov walk $\left( y + S_n\right) \mathbbm 1_{ \left\{ \tau_y > n \right\} }$ is non-negative, the random variable 
$ \left( y+\rho x + M_n\right) \mathbbm 1_{ \left\{ \tau_y > n \right\}  }$ may be not. 
In the case when $\bb E(a) \geq 0$ we take advantage of the properties presented in the next lemma.
\begin{lemma}\ 
\label{Mtauyprop}
\begin{enumerate}
	\item For all $x \in \bb R$ and $y>0$,
		\[
		y+\rho x+M_{\tau_y} \leq 0, \quad   \bb P_x\text{-a.s.}
	\]
	\item For all $x \in \bb R$ and $y>0$,
		\[
		\frac{X_{\tau_y}}{1-\bb E(a)} < y + \rho x+M_{\tau_y}, \quad \bb P_x\text{-a.s.}
	\]
	\item For all $x \in \bb R$ and $y>0$, the sequence $\left( (y+\rho x +M_n) \mathbbm{1}_{\{\tau_y>n\}} \right)_{n\geq 0}$ is a submartingale with respect to $\bb P_x$.
\end{enumerate}
\end{lemma}

\begin{proof}
\textit{Claim 1.} 
Let, for brevity,  $z=y+\rho x.$ 
Since, by the definition of $\tau_y$,
	\[
	X_{\tau_y} = y+S_{\tau_y} - (y + S_{\tau_y-1}) <0,
\]
it follows from (\ref{MfSetX}) and the bound $\bb E(a) \geq 0$ that $z+M_{\tau_y} \leq y+ S_{\tau_y} \leq 0$.

\textit{Claim 2.} Rewrite the martingale representation (\ref{MfSetX}) in the form
\begin{equation}
\label{MetSn-1}
	z+M_n = y + S_{n-1} + \frac{X_n}{1-\bb E(a)}.
\end{equation}
So, at the exit time $\tau_y$,
	\[
	\frac{X_{\tau_y}}{1-\bb E(a)} = z+M_{\tau_y} - \left( y + S_{\tau_y-1} \right) < z+M_{\tau_y}.
\]

\textit{Claim 3.} Using the first claim and the fact that $(M_n)_{n\geq 0}$ is a martingale,
\begin{align*}
	\bb E_x\left( \sachant{z+M_{n+1} \,;\, \tau_y>n+1}{\mathscr{F}_n} \right) &= z+M_{n} - \bb E_x\left( \sachant{z+M_{\tau_y} \,;\, \tau_y=n+1}{\mathscr{F}_n} \right)\\
	&\hspace{2.27cm} - \bb E_x\left( \sachant{z+M_{n+1}}{\mathscr{F}_n} \right)\mathbbm{1}_{\{\tau_y\leq n\}} \\
	&\geq (z+M_n)\mathbbm{1}_{\{\tau_y> n\}}.
\end{align*}
\end{proof}

In the next lemma we obtain a first bound for the expectation of the killed martingale $( (y + \rho x + M_n) \mathbbm 1_{\{ \tau_y > n \}} )_{n\geq 0}$ 
which is of order $n^{1/2-2\ee}$, for some $\ee >0$. Using a recurrent procedure we improve it subsequently to a bound not depending on $n$. 

\begin{lemma}
\label{firstupperbound}
Let $p\in (2,\alpha)$. For any $\ee \in (0, \frac{p-2}{4p})$, $x\in \bb R$, $y>0$ and $n \in \bb N$, we have
	\[ 
	\bb E_x\left( y + \rho x +M_n \,;\, \tau_y>n \right) \leq y+\rho x + c \abs{x} + c_p n^{1/2-2\ee}.
\]
\end{lemma}

\begin{proof}
By the Doob optional stopping theorem and the claim 2 of Lemma \ref{Mtauyprop}, with $z=y + \rho x,$
	\[
	\bb E_x\left( z+M_n \,;\, \tau_y >n \right) \leq z -\bb E_x\left( \frac{X_{\tau_y}}{1-\bb E(a)}  \,;\, \tau_y \leq n \right).
\]
Note that $X_{n} = \prod_{i=1}^n a_i x + X_n^0$, with $X_n^0$ given by (\ref{defdeXn0}). Then, with $\ee \in(0,1/4),$
\begin{align*}
	\bb E_x&\left( z+M_n \,;\, \tau_y >n \right) \\
	\leq\; & z +c\sum_{k=1}^n \prod_{i=1}^{k} \bb E\left( \abs{a_i} \right) \abs{x} + c\bb E_x\left( \abs{X_{\tau_y}^0}  \,;\, \tau_y \leq n \,,\, \underset{1 \leq k \leq n}{\max} \abs{X_k^0} \leq n^{1/2-2\ee}\right) \\
	& + c\bb E_x\left( \abs{X_{\tau_y}^0}  \,;\, \tau_y \leq n \,,\, \underset{1 \leq k \leq n}{\max} \abs{X_k^0} > n^{1/2-2\ee}\right).
\end{align*}
By the Markov inequality, for $2<p<\alpha$,
	\[
	\bb E_x\left( z+M_n \,;\, \tau_y >n \right) \leq z +c\sum_{k=1}^n \bb E^k\left(\abs{a}\right) \abs{x} + c n^{1/2-2\ee} + c\bb E_x\left( \frac{\underset{1 \leq k \leq n}{\max} \abs{X_k^0}^p}{n^{\frac{p-1}{2} (1-4\ee)  }} \right).
\]
By Lemma \ref{MCM} (with $x=0$),
	\[
	\bb E_x\left( z+M_n \,;\, \tau_y >n \right) \leq z +c \abs{x} + c n^{1/2-2\ee} + c_p\frac{n}{n^{\frac{p-1}{2} (1-4\ee)  }} .
\]
Choosing $\ee$ small enough to have ${\frac{p-1}{2} (1-4\ee)  } > 1/2 + 2\ee$, concludes the proof.
\end{proof}

Now we give an improvement of Lemma \ref{firstupperbound} which establishes a bound of the expectation of the killed martingale $((y+\rho x+M_n) \mathbbm 1_{\{\tau_y > n \}})_{n\geq 0}$ depending only on the starting values $x, y$. 

\begin{lemma}
\label{intofthemartcond}
For any $\delta >0$, $p \in (2,\alpha)$, $x\in \bb R,$ $y>0$ and $n \geq 0$,
\begin{align*}
\bb E_x \left( y + \rho x+M_n \,;\, \tau_y > n \right) &\leq \left( 1+c_{p} \delta \left(1+\abs{x}\right)^{p-1} \right) y + c_{p,\delta}\left(1+\abs{x}\right)^{p}. 
\end{align*}
Moreover, with $\delta=1$, for any $p \in (2,\alpha)$, $x\in \bb R,$ $y>0$ and $n \geq 0$,
\begin{align*}
\bb E_x \left( y + \rho x+M_n \,;\, \tau_y > n \right) \leq c_{p}\left( 1+y+\abs{x} \right)\left( 1+\abs{x} \right)^{p-1}.
\end{align*}

\end{lemma}
\begin{proof}
Let $\ee \in (0, \ee_1]$, where $\ee_1=\min \left(\ee_0, \frac{p-2}{4p}   \right)$ and $\ee_0$ is defined in Lemma \ref{concentnu}. 
Set $z= y+\rho x$. 
Assume first that $y> n^{1/2-\ee}$. From Lemma \ref{firstupperbound}, we deduce that,
	\[
	\bb E_x\left( y+\rho x+M_n \,;\, \tau_y>n \right) \leq y+\rho x + c \abs{x} + c_p n^{1/2-2\ee} \leq (1+c_p n^{-\ee}) y + c \abs{x},
\]
which proves the lemma when $y > n^{1/2-\ee}$ and $n $ is larger than $\delta^{-1/\ee}$.

Now, we turn to the case $0< y \leq n^{1/2-\ee}$. Introduce the following stopping time:
	\[
	\nu_n^\ee = \nu_n+\pent{n^\ee}.
\]
We have the following obvious decomposition: 
\begin{align}
 \bb E_x&\left( z+M_n \,;\, \tau_y>n \right) \nonumber\\
&= \underbrace{\bb E_x\left( z+M_n \,;\, \tau_y>n \,,\, \nu_n^\ee > \pent{n^{1-\ee}} \right)}_{=:J_1} + \underbrace{\bb E_x\left( z+M_n \,;\, \tau_y>n \,,\, \nu_n^\ee \leq \pent{n^{1-\ee}} \right)}_{=:J_2}. \label{J1et2001}
\end{align}

\textit{Bound of $J_1$.} Using the H\"older inequality for $1 < p < \alpha$, Lemma \ref{majmartx} and Lemma \ref{concentnu}, we have
	\[
	J_1 \leq c_{p,\ee} \sqrt{n}  \left( 1+ y + \abs{x} \right) \frac{\left(1+\abs{x} \right)^{p-1}}{n^ {(p-1)(\frac{1}{2}-\ee)}}.
\]
As $\ee < \frac{p-2}{4p}$, denoting $C_{p,\ee}(x,y)= c_{p,\ee}\left( 1+ y + \abs{x} \right)\left(1+\abs{x} \right)^{p-1}$, for all $n \geq 1$,
\begin{equation}
 \label{J2002} J_1 \leq \frac{C_{p,\ee}(x,y)}{n^\ee}.
\end{equation}

\textit{Bound of $J_2$.} Using the martingale representation (\ref{MfSetX}) for the Markov walk $(y+S_n)_{n\geq 1}$, by the Markov property,
\begin{align*}
	J_2 &= \sum_{k=1}^{\pent{n^{1-\ee}}} \int_{\bb R \times \bb R_+^*} E_{x'} \left( y' + \rho x' + M_{n-k} \,;\, \tau_{y'} > n-k \right)\\
	&\hspace{4cm}	\times \bb P_{x} \left( X_{\nu_n^\ee} \in \dd x' \,,\, y+S_{\nu_n^\ee} \in \dd y' \,,\, \tau_y > \nu_n^\ee \,,\, \nu_n^\ee =k \right).
\end{align*} 
By Lemma \ref{firstupperbound},
\begin{align*}
	J_2 \leq\; &\bb E_{x} \left( z+M_{\nu_n^\ee}   + c\abs{X_{\nu_n+\pent{n^{\ee}}}} + c_p n^{1/2-2\ee} \,;\, \tau_y > \nu_n^\ee \,,\, \nu_n^\ee \leq \pent{n^{1-\ee}} \right).
\end{align*}
For the term $z+M_{\nu_n^\ee}$, we use the fact that $((z+M_n) \mathbbm{1}_{\{\tau_y>n\}} )_{n\geq 0}$ is a submartingale (claim 3 of Lemma \ref{Mtauyprop}), while for the term $c\abs{X_{\nu_n+\pent{n^{\ee}}}}$ we use the Markov property at $\nu_n$ and Lemma \ref{MCM}. This gives
\begin{align*}
	J_2 \leq\; & \bb E_x \left( z+M_{\pent{n^{1-\ee}}} \,;\, \tau_y > \pent{n^{1-\ee}} \,,\, \nu_n^\ee \leq \pent{n^{1-\ee}} \right) \\
	& + c_p \bb E_{x} \left( n^{1/2-2\ee} + \bb E^{\pent{n^{\ee}}} \left(\abs{a} \right) \abs{X_{\nu_n}} \,;\, \tau_y > \nu_n \,,\, \nu_n \leq \pent{n^{1-\ee}} \right).
\end{align*}
Since $0< y \leq n^{1/2-\ee}$ and $\nu_n$ is the first time when $(\abs{y+S_n})_{n \geq 1}$ exceeds $n^{1/2-\ee}$, the jump $X_{\nu_n} $ is necessarily positive on the event $\{ \tau_y> \nu_n \}$. 
  Therefore, under the condition $\bb E(a) \geq 0$, by the representation (\ref{MfSetX}) we have $z+M_{\nu_n} > n^{1/2-\ee}$. Using the last bound, we obtain
\begin{align*}
J_2 \leq\; & \bb E_x \left( z+M_{\pent{n^{1-\ee}}} \,;\, \tau_y > \pent{n^{1-\ee}} \,,\, \nu_n^\ee \leq \pent{n^{1-\ee}} \right) \\
&+ c_p \bb E_{x} \left( \frac{z+M_{\nu_n}}{n^{\ee}} \,;\, \tau_y > \nu_n \,,\, \nu_n \leq \pent{n^{1-\ee}} \right) + \e^{-c_p n^\ee} \bb E_{x} \left( \abs{X_{\nu_n}} \,;\, \nu_n \leq \pent{n^{1-\ee}} \right).
\end{align*}
Again, using the fact that $\left((z+M_n) \mathbbm{1}_{\{\tau_y>n\}} \right)_{n \geq 0}$ is a submartingale and Lemma \ref{MCM},
we bound $J_2$ as follows,
\begin{align}
 J_2 \leq\; &\left(1+ \frac{c_p}{n^\ee} \right) \bb E_x \left( z+M_{\pent{n^{1-\ee}}} \,;\, \tau_y > \pent{n^{1-\ee}} \right) + \e^{-c_{p,\ee} n^\ee} n^{1-\ee} \left( 1+ \abs{x} \right) \nonumber\\
&-\underbrace{\bb E_x \left( \left(z+M_{\pent{n^{1-\ee}}}\right)\left( \mathbbm{1}_{\left\{\nu_n^\ee > \pent{n^{1-\ee}} \right\}} + \frac{c_p}{n^\ee} \mathbbm{1}_{\left\{\nu_n > \pent{n^{1-\ee}} \right\}} \right) \,;\, \tau_y > \pent{n^{1-\ee}} \right)}_{=:J_3}. \label{petitJ3}
\end{align}
We bound $J_3$ in a same manner as $J_1$,
	\[
	\abs{J_3} \leq c_{p,\ee} \sqrt{\pent{n^{1-\ee}}} \left( 1 + y + \abs{x} \right) c_{p,\ee} \frac{\left(1+\abs{x} \right)^{p-1}}{n^{\frac{p-1}{2}-(p-1)\ee}} \leq \frac{C_{p,\ee}(x,y)}{n^{\ee}}.
\]
Inserting this bound in (\ref{petitJ3}) and using (\ref{J2002}) and (\ref{J1et2001}) we find that
\[
 \bb E_x\left( z+M_n \,;\, \tau_y>n \right) \leq \left(1+ \frac{c_p}{n^\ee} \right) \bb E_x \left( z+M_{\pent{n^{1-\ee}}} \,;\, \tau_y > \pent{n^{1-\ee}} \right) 
 + \frac{C_{p,\ee}(x,y)}{n^\ee}.
\]
Since $( (z + M_n) \mathbbm 1_{\left\{\tau_y >n\right\}} )_{n\geq 0}$ is a submartingale, 
the sequence $u_n = \bb E_x\left( z+M_n \,;\, \tau_y>n \right)$ is non-decreasing.
By Lemma \ref{lemanalyse} used with $\alpha=c_p$, $\beta = C_{p,\ee}(x,y)$ and $\gamma=0$ it follows that
\[
\bb E_x\left( z+M_n \,;\, \tau_y>n \right) \leq \left( 1+ \frac{c_{p,\ee}}{n_f^\ee} \right) \bb E_x \left( z+M_{n_f} \,;\, \tau_y > n_f \right)
+ \frac{C_{p,\ee} (x,y)}{n_f^\ee}.
\]
By Lemma \ref{majmartx} and the fact that $z=y+\rho x$, we have
\begin{align*}
\bb E_x\left( z+M_n \,;\, \tau_y>n \right) &\leq \left( 1+\frac{c_{p,\ee}}{n_f^{\ee}} \right) y + c_{p,\ee} \sqrt{n_f} + c_{p,\ee} \abs{x} \\
 &  + \frac{c_{p,\ee}}{n_f^{\ee}} \left( 1+y+\abs{x}\right)\left(1+\abs{x}\right)^{p-1} \\
 &\leq  \left( 1+\frac{c_{p,\ee} \left(1+\abs{x}\right)^{p-1}}{n_f^{\ee}} \right) y + c_{p,\ee,n_f}\left(1+\abs{x}\right)^{p}.
\end{align*}
Choosing $n_f \geq \delta^{-1/\ee} $ concludes the proof of the lemma when $n\geq \delta^{-1/\ee}$.

Now, when $n\leq \delta^{-1/\ee}$, a bound of $\bb E_x \left( z+M_n \,;\, \tau_y > n \right)$ is obtained immediately from Lemma \ref{majmartx}: since $z=y+\rho x$, for any $y>0$,
\begin{align*}
	\bb E_x \left( z+M_n \,;\, \tau_y > n \right) &\leq y + c \abs{x} + \bb E_x \left( \abs{M_n} \right) \leq y + c \abs{x} + c \sqrt{n} \leq y + c_{\delta} \left( 1+\abs{x} \right),
\end{align*}
and we conclude that the lemma holds true for any $n \in \bb N$.

\end{proof}

We can now transfer the bound provided by Lemma \ref{intofthemartcond} to the Markov walk $(y+S_n)_{n\geq 0}$.
\begin{corollary}
\label{intofrandwalk}
For any $p \in (2,\alpha)$, $x\in \bb R,$ $y>0$ and $n\in \bb N$,
	\[
	\bb E_x \left( y+S_n \,;\, \tau_y > n \right) \leq c_p \left( 1+y+\abs{x} \right)\left( 1+\abs{x} \right)^{p-1}.
\]
\end{corollary}

\begin{proof}
Using equation (\ref{MfSetX}), the results follows from Lemma \ref{intofthemartcond} and Lemma \ref{MCM}.
\end{proof}

\subsection{Integrability of the killed martingale: the case $\bb E(a)<0$}
\label{CMWIcasneg}

We adapt the mainstream of the proof for the case $\bb E(a) \geq 0$ given in the previous section, highlighting the details that have to be modified.
 
In the discussion preceding Lemma \ref{Mtauyprop}, we noted that  $\left( y+ \rho x + M_n\right) \mathbbm 1_{ \left\{ \tau_y > n \right\}  }$ may not be positive.  
In the case $\bb E(a)<0$,
we overcome this by introducing the exit time of the martingale $\left( y+\rho x+M_n \right)_{n\geq 0}$: for any $y \in \bb R$,
	\[
	T_y = \min \{ k \geq 1,\,\, y+\rho x+M_k \leq 0 \}.
\]
By Corollary \ref{Exitfinit} we have $\bb P_x \left(T_y < +\infty\right) = 1$ for any $x \in \bb R.$     
The main point is to show the integrability of $y+\rho x+M_{T_y}$. Under the assumption  $\bb E(a)<0$  we have $\tau_y \leq T_y$, which 
together with the fact $\left( \abs{y+\rho x+M_n} \right)_{n\geq 0}$ is a submartingale, implies that $y+\rho x+M_{\tau_y}$ is integrable.

\begin{lemma}
\label{MTyprop} \ 
\begin{enumerate}
	\item For all $x \in \bb R$ and $y>0$,
	  \[
		\tau_y \leq T_y \quad \bb P_x \text{-a.s.}
	\]
	\item For all $x \in \bb R$ and $y \in \bb R$, the sequence $\left( (y+\rho x+M_n) \mathbbm{1}_{\{T_y>n\}} \right)_{n\geq 0}$ is a submartingale with respect to $\bb P_x$.
\end{enumerate}
\end{lemma}

\begin{proof} \textit{Claim 1.}
We note that when $T_y > 1$, by (\ref{MfSetX}) and (\ref{MetSn-1}), with $z=y+\rho x$,
\begin{align*}
	y+S_{T_y} &= z + M_{T_y} - \rho X_{T_y} \leq - \rho X_{T_y}, \\
	y+S_{T_y-1} &= z + M_{T_y} - \frac{X_{T_y}}{1-\bb E(a)} \leq - \frac{X_{T_y}}{1-\bb E(a)}.
\end{align*}
Since $\rho <0 $, according to the positivity or non-positivity of $X_{T_y}$, we have respectively $y+S_{T_y} \leq 0$ or $y+S_{T_y-1} \leq 0$. When $T_y = 1$ and $y>0$ we have $X_1 <0$ and so $\tau_y = 1 = T_y$.

\textit{Claim 2.} In a same manner as in the proof of the claim 3 of Lemma \ref{Mtauyprop}, the claim 2 is a consequence of the fact that $z+M_{T_y} \leq 0$ and that $(M_n)_{n\geq 0}$ is a martingale.
\end{proof}

The following lemma is similar to Lemma \ref{firstupperbound} but with $T_y$ replacing $\tau_y$.
\begin{lemma}
\label{firstupperboundcasneg}
Let $p \in (2,\alpha)$. For any $\ee \in (0, \frac{p-2}{4p})$, $x\in \bb R$, $y>-\rho x$ and $n \geq 0$, we have
	\[
	\bb E_x\left( y+\rho x+M_n \,;\, T_y>n \right) \leq y+ \rho x + c \abs{x} + c_p n^{1/2-2\ee}.
\]
\end{lemma}

\begin{proof} Note that $z=y + \rho x >0$. Since at the exit time $T_y$ we have $0 \geq z + M_{T_y} \geq \xi_{T_y} = \frac{X_{T_y}- \bb E(a) X_{T_y-1}}{1-\bb E(a)}$, by the Doob optional stopping theorem,
	\[
	\bb E_x\left( z+M_n \,;\, T_y >n \right) \leq z + c \bb E_x\left( \abs{X_{T_y}} + \abs{X_{T_y-1}}  \,;\, T_y \leq n \right).
\]
Since  $\abs{X_{T_y}} + \abs{X_{T_y-1}} \leq  2 \max_{1 \leq k \leq n} \abs{X_k} +\abs{x}$ on $\left\{T_y\leq n \right\}$,  following the proof of Lemma \ref{firstupperbound},
\begin{align*}
	 \bb E_x \left( z+M_n \,;\, T_y >n \right)   \leq z & + c \left(1+  \sum_{k=1}^n \prod_{i=1}^k \bb E \left( \abs{a_i} \right)\right) \abs{x} 
	\\ & +c n^{1/2-2\ee}   \bb P\left( \underset{1 \leq k \leq n}{\max} \abs{X_{k}^0} \leq n^{1/2-2\ee} \right) \\
	&+ c \bb E\left( \underset{1 \leq k \leq n}{\max} \abs{X_{k}^0}  \,;\, \underset{1 \leq k \leq n}{\max} \abs{X_{k}^0} > n^{1/2-2\ee} \right) \\
	 \leq  z &+ c \abs{x} + c_p n^{1/2-2\ee}.
\end{align*}
\end{proof}

\begin{lemma}
\label{intofthemartcondcasneg}
Let $p \in (2,\alpha)$. There exists $\ee_1>0$ such that for any  $\ee \in (0, \ee_1),$ $x\in \bb R$, $y\in \bb R$, $n\geq 0 $ and $2\leq n_f\leq n$,
\begin{align*}
\bb E_x \left( y+\rho x+M_n \,;\, T_y > n \right) &\leq \left( 1+ \frac{c_{p,\ee}}{n_f^\ee} \right) \max(y,0) + c_{p,\ee} \abs{x} + c_{p,\ee} \sqrt{n_f} + \e^{-c_{p,\ee} n_f^{\ee}} \abs{x}^p \\
&\leq c_p \left( 1+\max(y,0)+\abs{x}^p \right).
\end{align*}
\end{lemma}

\begin{proof}
We proceed as in the proof of Lemma \ref{intofthemartcond}. Set $\ee_1=\min \left(\ee_0, \frac{p-2}{4p} \right)$, where $\ee_0$ is defined in Lemma \ref{concentnu}. Let $\ee \in (0, \ee_1]$.
With $z=y+\rho x$ and $v_n^\ee = v_n + \pent{n^\ee}$, we have
\begin{align}
	\bb E_x\left( z+M_n \,;\, T_y>n \right) =\; &\underbrace{\bb E_x\left( z+M_n \,;\, T_y>n \,,\, v_n^\ee > \pent{n^{1-\ee}} \right)}_{=:J_1} \nonumber \\
	& +\underbrace{\bb E_x\left( z+M_n \,;\, T_y>n \,,\, v_n^\ee \leq \pent{n^{1-\ee}} \right)}_{=:J_2}. \label{J1et2001casneg}
\end{align}

\textit{Bound of $J_1$.}
Let $m_{\ee} = \pent{n^{1-\ee}}-\pent{n^\ee}$.  Since on $\{v_n> m_{\ee}\}$ it holds $z'=z+M_{m_\ee} \leq n^{1/2-\ee}$, by the Markov property we write that
	\[
	J_1 \leq n^{1/2-\ee} \bb P_x\left( v_n > m_\ee \right) + \int_{\bb R} \bb E_{x'} \left( \abs{M_{n-m_{\ee}}} \right) \bb P_x \left( X_{m_\ee} \in \dd x' \,,\, v_n > m_\ee \right).
\]
By Lemma \ref{majmartx} and the H\"older inequality,
\begin{align*}
	J_1 &\leq n^{1/2-\ee} \bb P_x\left( v_n > m_\ee \right) + \bb E_x\left( c \left( \sqrt{n-m_\ee} + \abs{X_{m_\ee}} \right) \,;\, v_n > m_\ee \right)\\
	&\leq c n^{1/2} \bb P_x\left( v_n > m_\ee \right) + \bb E_x^{1/p} \left( \abs{X_{m_\ee}}^p \right) \bb P_x^{1/q} \left( v_n > m_\ee \right).
\end{align*}
By Lemma \ref{MCM} and Lemma \ref{concentnu} (since $m_{\ee} \geq n^{1-\ee}/c_\ee$),
\begin{equation}
 \label{J2002casneg}
 J_1 \leq \frac{c_{p,\ee}}{ n^{\frac{p-1}{2}-(p-1)\ee} } + \e^{-c_{p,\ee} n^{1-2\ee}} \abs{x}^p.
\end{equation}

\textit{Bound of $J_2$.}  
Repeating the arguments used for bounding the term $J_2$ in Lemma \ref{intofthemartcond}, by the Markov property and Lemma \ref{firstupperboundcasneg}, we get
\begin{align*}
	J_2 &\leq \bb E_{x} \left( z+M_{v_n^\ee} + c \abs{X_{v_n^\ee}} + c_p n^{1/2-2\ee} \,;\, T_y > v_n^\ee \,,\, v_n^\ee \leq \pent{n^{1-\ee}} \right).
\end{align*}
Using the claim 2 of Lemma \ref{MTyprop} and Lemma \ref{MCM},
\begin{align*}
	J_2 \leq\; & \bb E_x \left( z+M_{\pent{n^{1-\ee}}} \,;\, T_y > \pent{n^{1-\ee}} \,,\, v_n^\ee \leq \pent{n^{1-\ee}} \right) \\
	&+ \bb E_{x} \left( c_p n^{1/2-2\ee} \,;\, T_y > v_n \,,\, v_n \leq \pent{n^{1-\ee}} \right) + \e^{-c_\ee n^\ee} \bb E_{x} \left( \abs{X_{v_n}} \,;\, v_n \leq \pent{n^{1-\ee}} \right).
\end{align*}
On the event $\{ T_y> v_n \}$, we have $n^{1/2-\ee} < z+M_{v_n}$. Hence
\begin{align*}
	J_2 \leq\; & \bb E_x \left( z+M_{\pent{n^{1-\ee}}} \,;\, T_y > \pent{n^{1-\ee}} \,,\, v_n^\ee \leq \pent{n^{1-\ee}} \right) \\
	&+ c_p \bb E_{x} \left( \frac{z+M_{v_n}}{n^{\ee}} \,;\, T_y > v_n \,,\, v_n \leq \pent{n^{1-\ee}} \right) + \e^{-c_\ee n^\ee} \bb E_{x} \left( \abs{X_{v_n}} \,;\, v_n \leq \pent{n^{1-\ee}} \right).
\end{align*}
Coupling this with (\ref{J2002casneg}) and (\ref{J1et2001casneg}) and using again the claim 2 of Lemma \ref{MTyprop}, we obtain that
\begin{align*}
  \bb E_x\left( z+M_n \,;\, T_y>n \right) &\leq \left(1+ \frac{c_p}{n^\ee} \right) \bb E_x \left( z+M_{\pent{n^{1-\ee}}} \,;\, T_y > \pent{n^{1-\ee}} \right) \\
	&\qquad + \frac{c_{p,\ee}}{ n^{\frac{p-1}{2}-(p-1)\ee} } + \e^{-c_{p,\ee} n^{\ee}} \abs{x}^p.
\end{align*}
Since $( (z+M_n) \mathbbm 1_{\{T_y>n\}} )_{n\geq 0}$ is a submartingale (claim 2 of Lemma \ref{MTyprop}), the sequence
$u_n = \bb E  (z+M_n) \mathbbm 1_{\{T_y>n\}}$ is non-decreasing. 
By Lemma \ref{lemanalyse} with $\alpha=c_p$, $\beta=c_{p,\ee}$, $\gamma=\abs{x}^p$ and $\delta=c_{p,\ee}$, we write that
\[
\bb E_x\left( z+M_n \,;\, T_y>n \right) \leq \left(1+ \frac{c_{p,\ee}}{n_f^\ee} \right) \bb E_x \left( z+M_{n_f} \,;\, T_y > n_f \right) + \frac{c_{p,\ee}}{ n_f^\ee } + \e^{-c_{p,\ee} n_f^{\ee}} \abs{x}^p.
\]
Using Lemma \ref{majmartx} and the fact that $z=y+\rho x$, we obtain that
	\[
	\bb E_x\left( z+M_n \,;\, T_y>n \right) \leq \left( 1+ \frac{c_{p,\ee}}{n_f^\ee} \right) \max(y,0) + c_{p,\ee} \abs{x} + c_{p,\ee} \sqrt{n_f} + \e^{-c_{p,\ee} n_f^{\ee}} \abs{x}^p.
\]
\end{proof}

\begin{corollary}
\label{intofrandwalkcasneg}
Let $p \in (2,\alpha)$. For any $x\in \bb R$, $y>0$ and $n\in \bb N$,
	\[
	\bb E_x \left( y+S_n \,;\, \tau_y > n \right) \leq c_p \left( 1+y+\abs{x}^p \right).
\]
\end{corollary}

\begin{proof}
By (\ref{MfSetX}) and the claim 1 of Lemma \ref{MTyprop}, we have
$$
\bb E_x \left( y+S_n \,;\, \tau_y > n \right) =  \bb E_x \left( y+\rho x+M_n \,;\,  T_y \geq \tau_y > n \right) - \bb E_x \left( \rho X_n \,;\, \tau_y > n \right). 
$$
The result follows from Lemma \ref{intofthemartcondcasneg}.
\end{proof}

\section{Existence of the harmonic function}
\label{Sec Harm Func}
In this section we prove Theorem \ref{ExofHaFu}. We split the proof into two parts according to the values of $\bb E(a) $.

\subsection{Existence of the harmonic function: the case $\bb E(a) \geq 0$}
\label{Sec Harm Func positif}

We start with the following assertion.
\begin{lemma}
\label{intdeMtau}
For any $x \in \bb R$ and $y>0$, the random variable $M_{\tau_y}$ is integrable. Moreover, for any $p \in (2,\alpha)$,
  \[
	\bb E_x \left( \abs{M_{\tau_y}} \right) \leq c_p \left( 1+y+\abs{x} \right)\left( 1+\abs{x} \right)^{p-1}.
\]
\end{lemma}

\begin{proof}
Let $z=y+\rho x$. Using the claim 1 of Lemma \ref{Mtauyprop} and the Doob optional stopping theorem, we have
\begin{align*}
 \bb E_x \left( \abs{M_{\tau_y}} \,;\, \tau_y \leq n \right) &\leq -\bb E_x \left( z+M_n \,;\, \tau_y \leq n \right) + y + \rho  \abs{x}\\
 &= \bb E_x \left( z+M_n \,;\, \tau_y > n \right)  - z + y + \rho  \abs{x}.
\end{align*}
By second bound in  Lemma \ref{intofthemartcond}, for all $n \geq 0$,
	\[
	\bb E_x \left( \abs{M_{\tau_y}} \,;\, \tau_y \leq n \right) \leq c_p \left( 1+y+\abs{x} \right)\left( 1+\abs{x} \right)^{p-1} =: C_p(x,y).
\]
Since $\left( \{ \tau_y \leq n \} \right)_{n \geq 1}$ is a non-decreasing sequence of events and  $\bb P_x \left(\tau_y < +\infty\right)=1$ for any $x \in \bb R$ (by Corollary \ref{Exitfinit}), the result follows by the Lebesgue monotone convergence theorem.
\end{proof}

Now, the claim 1 of Theorem \ref{ExofHaFu} concerning the existence of the function $V$ is a direct consequence of the previous lemma:
\begin{corollary}
For any $x \in \bb R$ and $y > 0$, the following function is well defined
	\[
	V(x,y) =  -\bb E_x \left( M_{\tau_y} \right).
\]
\end{corollary}

The following two propositions prove the claims 2 and 3 of Theorem \ref{ExofHaFu} under Conditions \ref{Mom001}, \ref{PosdeX1-1} and $\bb E (a) \geq 0$. 
\begin{proposition}
\ 
\label{Invfunctprop}
\begin{enumerate}
\item For any $x\in \bb R$ and $y>0$,
\begin{align*}
V(x,y) = \underset{n\to +\infty}{\lim} \bb E_x \left( y+\rho x +M_n \,;\, \tau_y > n \right) = \underset{n\to +\infty}{\lim} \bb E_x \left( y+S_n \,;\, \tau_y > n \right).
\end{align*}
\item For any $x \in \bb R$, the function $V(x,.)$ is non-decreasing.
\item For any $\delta >0$, $p \in (2,\alpha)$, $x\in \bb R$ and $y>0$,
	\[
	\max(0,y+\rho x) \leq V(x,y) \leq \left( 1+c_{p} \delta \left(1+\abs{x}\right)^{p-1} \right) y + c_{p,\delta}\left(1+\abs{x}\right)^{p}.
\]
\item For any $x \in \bb R$,
	\[
	\underset{y\to +\infty}{\lim} \frac{V(x,y)}{y} = 1.
\]
\end{enumerate}
\end{proposition}

\begin{proof} We use the notation  $z= y+\rho x$.

\textit{Claim 1.} Since, by Lemma \ref{intdeMtau}, $M_{\tau_y}$ is integrable, we have by the Lebesgue dominated convergence theorem,
	\[
	\bb E_x \left( z+M_n \,;\, \tau_y > n \right) = z - \bb E_x \left( z+M_{\tau_y} \,;\, \tau_y \leq n \right) \underset{n\to +\infty}{\longrightarrow} -\bb E_x \left( M_{\tau_y} \right) = V(x,y).
\]
To prove the second equality of the claim 1 we use Lemma \ref{MCM} and the fact that $\tau_y <+\infty$:
	\[
	\abs{\bb E_x \left( X_n \,;\, \tau_y > n \right)} \leq \bb E_x^{1/2} \left( \abs{X_n}^2 \right) \sqrt{ \bb P_x \left( \tau_y > n \right) } \leq c_2 \left(1+\abs{x}\right) \sqrt{ \bb P_x \left( \tau_y > n \right) } \underset{n\to +\infty}{\longrightarrow} 0.
\]
Using (\ref{MfSetX}), we obtain the claim 1.
 
\textit{Claim 2.} If $y_1\leq y_2$, then $\tau_{y_1} \leq \tau_{y_2}$ and
	\[
	\bb E_x \left( y_1+S_n \,;\, \tau_{y_1} > n \right) \leq \bb E_x \left( y_2+S_n \,;\, \tau_{y_1} > n \right) \leq \bb E_x \left( y_2+S_n \,;\, \tau_{y_2} > n \right).
\]
Taking the limit as $n\to +\infty$ we get the claim 2.

\textit{Claim 3.} The upper bound follows from
the claim 1 and Lemma \ref{intofthemartcond}. On the event $\{ \tau_y > n \}$, we obviously have $y+S_n>0$ and so by claim 1, $V(x,y) \geq 0$. Moreover, since $z+M_{\tau_y} \leq 0$ (by claim 1 of Lemma \ref{Mtauyprop}), we have, by claim 1,
	\[
	V(x,y) = z - \underset{n\to +\infty}{\lim} \mathbb E_x \left( z+M_{\tau_y} \,;\, \tau_y \leq n \right) \geq z,
\]
which proves the lower bound.

\textit{Claim 4.} By the claim 3, for all $\delta >0$, $x\in \bb R$,
	\[
	1 \leq \underset{y\to+\infty}{\rm liminf} \frac{V(x,y)}{y} \leq \underset{y\to+\infty}{\rm limsup} \frac{V(x,y)}{y} \leq \left( 1+c_{p} \delta \left(1+\abs{x}\right)^{p-1} \right).
\]
Letting $\delta \to 0$, we obtain the claim 4.
\end{proof}

We now prove that $V$ is harmonic on $\bb R \times \bb R_+^*$.

\begin{proposition}
\label{HaetposdeV}
\noindent
\begin{enumerate}
\item The function $V$ is $\bf{Q_+}$-harmonic on $\bb R \times \bb R_+^*$: for any $x\in \bb R$ and $y>0$,
	\[
	\mathbf{Q_+}V(x,y) = V(x,y).
\]
\item The function $V$ is positive on $\bb R \times \bb R_{+}^*$.
\end{enumerate}
\end{proposition}

\begin{proof}
\textit{Claim 1.}
Denote for brevity $V_n(x,y) = \bb E_x \left( y+S_n \,;\, \tau_y > n \right)$. For all $x\in \bb R$, $y>0$ and  $n \geq 1$, by the Markov property,
	\[
	V_{n+1}(x,y) = \bb E_x \left( V_n(X_1,y+S_1) \,;\, \tau_y > 1 \right).
\]
By Corollary \ref{intofrandwalk}, we see that the quantity $V_n(X_1,y+S_1)$ is dominated by the random variable $c_p \left( 1+y+S_1+\abs{X_1} \right)\left( 1+ \abs{X_1} \right)^{p-1}$ which is integrable with respect to $\bb E_x$. Consequently, by the Lebesgue dominated convergence theorem and the claim 1 of Proposition \ref{Invfunctprop},
	\[
	V(x,y) = \bb E_x \left( V(X_1,y+S_1) \,;\, \tau_y > 1 \right) = \textbf{Q}_+V(x,y),
\]
where by convention, $V(x,y)\mathbbm 1_{\{y > 0\}} = 0$ if $y \leq 0$ and $x \in \bb R.$

\textit{Claim 2.} Fix $x\in \bb R$ and $y>0$. Using the claim 1 and the fact that $V$ is non-negative on $\bb R \times \bb R_+^*$ (claim 3 of Lemma \ref{Invfunctprop}) we write
	\[
	V(x,y) \geq \bb E_x \left( V(X_1,\, y+S_1) \,;\, \tau_y > 1 \,,\, X_1 > \frac{-y}{2(1+\rho)} \right).
\]
By the lower bound of the claim 3 of Lemma \ref{Invfunctprop} and (\ref{MfSetX}),
	\[
	V(x,y) \geq \bb E_x \left( y+ (1+\rho) X_1 \,;\, \tau_y > 1 \,,\, X_1 > \frac{-y}{2(1+\rho)} \right) \geq \frac{y}{2} \bb P_x \left( X_1 > \frac{-y}{2(1+\rho)} \right).
\]
By Condition \ref{PosdeX1-1}, we conclude that, $V(x,y) > 0$ for any $x\in \bb R$ and $y>0$.
\end{proof}

\subsection{Existence of the harmonic function: the case $\bb E(a)<0$}
\label{Sec Harm Func cas neg}

In this section we prove the harmonicity and the positivity of the function $V$ in the case $\bb E(a)<0.$
The following assertion is the analogue of Lemma \ref{intdeMtau}.
\begin{lemma}
\label{intdeMtaucasneg}
The random variables $M_{T_y}$ and $M_{\tau_y}$ are integrable.
\begin{enumerate}
\item	For any $x \in \bb R$ and $y \in \bb R$,
	\[
	\bb E_x \left( \abs{M_{T_y}} \right) \leq c_p \left( 1+\abs{y}+\abs{x}^p \right).
\]
\item	For any $x\in \bb R$ and $y\in \bb R$,
	\[
	\bb E_x \left( \abs{M_{\tau_y}} \right) \leq c_p \left( 1+\abs{y}+\abs{x}^p \right).
\]
\end{enumerate}
\end{lemma}

\begin{proof} 
\textit{Claim 1.} The proof is similar to that of Lemma \ref{intdeMtau} using Lemma \ref{intofthemartcondcasneg} instead of Lemma \ref{intofthemartcond}  and the fact that by Corollary \ref{Exitfinit} we have $\bb P_x \left(T_y < +\infty \right)=1,$ $x\in \bb R$.

\textit{Claim 2.} By the claim 1 of Lemma \ref{MTyprop}, we have $\tau_y \wedge n \leq T_y \wedge n$. Since $\left( \abs{M_n} \right)_{n\geq 0}$ is a submartingale,
with $z=y+\rho x$,
\begin{align*}
	\bb E_x \left( \abs{M_{\tau_y}} \,;\, \tau_y \leq n \right) \leq \bb E_x \left( \abs{M_{\tau_y\wedge n}} \right) \leq \bb E_x \left( \abs{M_{T_y\wedge n}} \right) \leq 2\abs{z} + 2\bb E_x \left( \abs{M_{T_y}} \,;\, T_y \leq n \right). 
\end{align*}
The Lebesgue monotone convergence theorem implies the claim 2.
\end{proof}
As a direct consequence of this corollary we have:
\begin{corollary}
For any $x \in \bb R$ and $y > 0$, the following function is well defined
	\[
	V(x,y) =  -\bb E_x \left( M_{\tau_y} \right).
\]
\end{corollary}
Consider the function
	\[
	W(x,y) = -\mathbb E_x \left( M_{T_y} \right),
\]
which will be used in the proof of the positivity of the function $V$ on $\bb R \times \bb R_+^*$.  
By Corollary \ref{intdeMtaucasneg}, the function $W$ is well defined on $\bb R \times \bb R$.

\begin{proposition}
\ 
\label{InvfunctpropW}
\begin{enumerate}
\item For any $x\in \bb R$ and $y \in \bb R$,
	\[
	W(x,y) = \underset{n\to +\infty}{\lim} \bb E_x \left( y+\rho x+M_n \,;\, T_y > n \right).
\]
\item For any $x \in \bb R$, the function $W(x,.)$ is non-decreasing.
\item For any $p \in (2,\alpha)$, there exists $\ee_1>0$ such that for any $\ee \in (0, \ee_1]$, $n_f \geq 2$, $x\in \bb R$ and $y\in \bb R$,
	\[
	\max(0,y+\rho x) \leq W(x,y) \leq \left( 1+ \frac{c_{p,\ee}}{n_f^\ee} \right) \max(y,0) + c_{p,\ee} \abs{x} + c_{p,\ee} \sqrt{n_f} + \e^{-c_{p,\ee} n_f^{\ee}} \abs{x}^p. 
\]
\item For any $x \in \bb R$, 
	\[
	\underset{y\to +\infty}{\lim} \frac{W(x,y)}{y} = 1.
\]
\item For any $x \in \bb R$ and $y \in \bb R$, 
	\[
	W(x,y) = \bb E_x \left( W\left( X_1, y+S_1 \right) \,;\, T_y > 1 \right),
\]
and $\left( W( X_n, y+S_n) \mathbbm 1_{\{T_y > n\}} \right)_{n\geq 0}$ is a martingale.
\end{enumerate}
\end{proposition}

\begin{proof}
The proof is very close to that of Proposition \ref{Invfunctprop}. The upper bound of the claim 3 is obtained taking the limit as $n\to +\infty$ in Lemma \ref{intofthemartcondcasneg}. We prove the claim 4 taking the limit as $y\to +\infty$ and then as $n_f \to +\infty$ in the inequality of the claim 3. The proof of the claim 5 is the same as that of the claim 1 of Proposition \ref{HaetposdeV}.
\end{proof}

Moreover, we have the following proposition.
\begin{proposition}
\label{Invfunctpropcasnegpart1}\ 
\begin{enumerate}
\item For any $x\in \bb R$ and $y >0$,
	\[
	V(x,y) = \underset{n\to +\infty}{\lim} \bb E_x \left( y+\rho x+M_n \,;\, \tau_y > n \right) = \underset{n\to +\infty}{\lim} \bb E_x \left( y+S_n \,;\, \tau_y > n \right).
\]
\item For any $x \in \bb R$, the function $V(x,.)$ is non-decreasing.
\item For any $p \in  (2,\alpha),$ $\delta >0$, $x\in \bb R$ and $y >0$,
	\[
	0 \leq V(x,y) \leq W(x,y) \leq  \left( 1+ c_p \delta \right) y + c_{p,\delta} \left( 1+ \abs{x}^p \right).
\]
\item The function $V$ is $\bf{Q_+}$-harmonic on $\bb R \times \bb R_+^*$: for any $x\in \bb R$ and $y>0$,
	\[
	\mathbf{Q_+}V(x,y) =V(x,y)
\]
and $\left( V(X_n,y+S_n) \mathbbm 1_{\{\tau_y > n \}} \right)_{n\geq 0}$ is a martingale.
\end{enumerate}
\end{proposition}

\begin{proof}
The proofs of the claims 1, 2, 4 and of the lower bound of the claim 3, being similar to that of the previous proposition and of the Proposition  \ref{Invfunctprop}, is left to the reader. 
The upper bound of the claim 3 is a consequence of the fact that $\tau_y \leq T_y$ (claim 1 of Lemma \ref{MTyprop}): with $z=y+\rho x$,
\begin{align*}
	V(x,y) &= \underset{n\to +\infty}{\lim} \bb E_x \left( z+M_n \,;\, \tau_y > n \right) \\
	&\leq \underset{n\to +\infty}{\lim} \bb E_x \left( z+M_n \,;\, T_y > n \right) = W(x,y).
\end{align*}
\end{proof}

Our next goal is to prove that $V(x,y) \geq \max\left(0,(1-\delta) y -  c_{p,\delta} \left(1+\abs{x}^{p}\right)  \right)$ from which we will deduce the positivity of $V$. For this we make appropriate adjustments to the proof of Lemmas \ref{firstupperbound} and Lemma \ref{intofthemartcond} where the submartingale $((y+\rho x+M_n) \mathbbm 1_{\{T_y > n}\})_{n\geq 0}$ will be replaced by the supermartingale $( W\left(X_n,y+S_n\right) \mathbbm 1_{ \{\tau_y > n\} } )_{n\geq 0}$.
Instead of upper bounds in  Lemmas \ref{firstupperbound} and Lemma \ref{intofthemartcond} the following two lemmas establish lower bounds. 

\begin{lemma}
\label{firstlowerboundcasneg}
For any $p \in (2,\alpha)$, there exists $\ee_1>0$ such that for any $\ee \in (0, \ee_1]$, $x\in \bb R$, $y>0$ and $n \in \bb N$,
	\[
	\bb E_x\left( W(X_n,y+S_n) \,;\, \tau_y>n \right) \geq W(x,y) - c_{p,\ee} n^{1/2-2\ee} - c_{p,\ee} \abs{x}^p.
\]
\end{lemma}

\begin{proof}
By the claim 1 of Lemma \ref{MTyprop} and the claim 5 of Lemma \ref{InvfunctpropW}, as in the proof of Lemma \ref{firstupperbound},
	\[
	\bb E_x\left( W(X_n,y+S_n) \,;\, \tau_y>n \right) = W(x,y)	- \bb E_x\left( W(X_{\tau_y},y+S_{\tau_y}) \,;\, T_y>\tau_y \,,\, \tau_y \leq n \right).
\]
Using the claim 3 of Proposition \ref{InvfunctpropW} and the fact that $y+S_{\tau_y} \leq 0$,
\begin{align*}
	\bb E_x&\left( W(X_{\tau_y},y+S_{\tau_y}) \,;\, T_y>\tau_y \,,\, \tau_y \leq n \right) \leq \\
	&\hspace{4cm} \bb E_x\left( c_{p,\ee} \abs{X_{\tau_y}} + c_{p,\ee} \sqrt{n_f} + \e^{-c_{p,\ee} n_f^{\ee}} \abs{X_{\tau_y}}^p \,;\, \tau_y \leq n \right).
\end{align*}
Taking $n_f = \pent{n^{1-4\ee}}$, the end of the proof is the same as the proof of Lemma \ref{firstupperbound}.
\end{proof}

\begin{lemma}
\label{lowerboundcasneg}
For any $p\in (2,\alpha)$ there exists $\ee_1 >0$ such that for any $\ee \in (0, \ee_1]$, $n_f \geq 2$, $x\in \bb R$ and $y>0$,
	\[
	\bb E_x\left( W\left(X_n,y+S_n\right) \,;\, \tau_y>n \right) \geq y \left( 1 - \frac{c_{p,\ee}}{n_f^\ee} \right)  - c_{p,\ee} n_f^2 \left( 1+\abs{x}^p \right).
\]
\end{lemma}

\begin{proof}
The proof is similar to that of Lemma \ref{intofthemartcond}. With $v_n^\ee = v_n+\pent{n^{\ee}}$, we have
	\[
	J_0 = \bb E_x \left( W\left(X_n,y+S_n\right) \,;\, \tau_y>n \right) \geq \bb E_x \left( W\left(X_n,y+S_n\right) \,;\, \tau_y>n \,,\, v_n^\ee \leq \pent{n^{1-\ee}} \right).
\]
Using the Markov property, Lemma \ref{firstlowerboundcasneg} and the fact that $n-v_n^\ee \leq n$,
	\[
	J_0 \geq \; \bb E_{x} \left( W\left( X_{v_n^\ee},y+S_{v_n^\ee} \right) -c_{p,\ee} n^{1/2-2\ee} - c_{p,\ee} \abs{X_{v_n^\ee}}^p \,;\, \tau_y > v_n^\ee \,,\, v_n^\ee \leq \pent{n^{1-\ee}} \right).
\]
By the claim 1 of Lemma \ref{MTyprop}, on $\{ \tau_y > v_n \}$ we have $z+M_{v_n} > n^{1/2-\ee}$, where $z=y+\rho x$. 
From this 
and the fact that $\left( W\left(X_n, y+S_n \right) \mathbbm{1}_{\{ \tau_y>n\}} \right)_{n\geq 1}$ is a supermartingale, 
as in the the bound of the the term $J_2$ of Lemma \ref{intofthemartcond}, 
we obtain that
\begin{align}
	J_0 \geq \; &\bb E_{x} \left( W\left( X_{\pent{n^{1-\ee}}},y+S_{\pent{n^{1-\ee}}} \right) \,;\, \tau_y > \pent{n^{1-\ee}} \right) \nonumber \\
	&- \bb E_{x} \left( W\left( X_{\pent{n^{1-\ee}}},y+S_{\pent{n^{1-\ee}}} \right) \,;\, \tau_y > \pent{n^{1-\ee}} \,,\, v_n^\ee > \pent{n^{1-\ee}} \right) \label{Bound-J0-001} \\
	&-\frac{c_{p,\ee}}{n^{\ee}} \bb E_{x} \left( z+M_{v_n} \,;\, T_y > v_n \,,\, v_n \leq \pent{n^{1-\ee}} \right) - \e^{-c_{p,\ee} n^\ee} \left( 1 + \abs{x}^p \right). \nonumber
\end{align}
Using the claim 3 of Proposition \ref{InvfunctpropW} with $n_f=n$ and the martingale representation (\ref{MfSetX}), the absolute value of the second term in the r.h.s.\ of (\ref{Bound-J0-001}) 
does not exceed
\begin{align*}
	& c_{p,\ee} \bb E_{x} \left( z+M_{\pent{n^{1-\ee}}} + \sqrt{n} + \abs{X_{\pent{n^{1-\ee}}}}     +\e^{-c_{p,\ee} n^\ee} \abs{X_{\pent{n^{1-\ee}}}}^p      \,;   \right. \\
	&\hspace{8cm} \left. \, \tau_y > \pent{n^{1-\ee}} \,,\, v_n^\ee > \pent{n^{1-\ee}} \right).
\end{align*}
Since $\left( (z+M_n) \mathbbm 1 _{\{ T_y > n\} } \right)_{n\geq 0}$ is a submartingale, by claim 2 of Lemma \ref{MTyprop},  the absolute value of the third term is less than 
\begin{align*}
	\frac{c_{p,\ee}}{n^{\ee}} \bb E_{x} \left( z+M_n \,;\, T_y > n \right).
\end{align*}
These bounds imply
\begin{align}
	J_0 \geq \; &\bb E_{x} \left( W\left( X_{\pent{n^{1-\ee}}},y+S_{\pent{n^{1-\ee}}} \right) \,;\, \tau_y > \pent{n^{1-\ee}} \right) \nonumber \\
	&- c_{p,\ee} \bb E_{x} \left( z+M_{\pent{n^{1-\ee}}} + \sqrt{n} + \abs{X_{\pent{n^{1-\ee}}}}   \,;\, \tau_y > \pent{n^{1-\ee}} \,,\, v_n^\ee > \pent{n^{1-\ee}} \right) \nonumber \\
	&-\e^{-c_{p,\ee} n^\ee} \bb E_{x} \left( \abs{X_{\pent{n^{1-\ee}}}}^p \,;\, \tau_y > \pent{n^{1-\ee}} \,,\, v_n^\ee > \pent{n^{1-\ee}} \right) \label{Bound-J0-003} \\
	&-\frac{c_{p,\ee}}{n^{\ee}} \bb E_{x} \left( z+M_n \,;\, T_y > n \right) - \e^{-c_{p,\ee} n^\ee} \left( 1 + \abs{x}^p \right). \nonumber
\end{align}
Using the Markov property with the intermediate time $m_\ee= \pent{n^{1-\ee}} - \pent{n^\ee}$, Lemmas  \ref{majmartx} and \ref{MCM}
and the fact that $v_n^\ee = v_n+\pent{n^\ee}$, 
the absolute value of the second term in the r.h.s. of (\ref{Bound-J0-003}) is bounded by
\begin{equation*}
c_{p,\ee}\bb E_{x} \left( \abs{   z+M_{m_\ee}  } + c n^{\ee/2}  + c \abs{X_{m_\ee}} + \sqrt{n} + c (1+ \abs{X_{m_\ee} })  \,;\, \tau_y > m_\ee \,,\, v_n > m_\ee \right),
\label{bound-term2-001}
\end{equation*}
which, in turn, using the fact that $z+M_{m_\ee} \leq n^{1/2-\ee}$ on $\{ v_n > m_\ee \}$, is less than
	\[
	c_{p,\ee}\bb E_{x} \left( \sqrt{n}  + \abs{X_{m_\ee}} \,;\, \tau_y > m_\ee \,,\, v_n > m_\ee \right).
\]
The absolute value of the third term in the r.h.s.\ of (\ref{Bound-J0-003}) is bounded using Lemma \ref{MCM} by
$\e^{-c_{p,\ee} n^\ee} \left( 1 + \abs{x}^p \right).$
The fourth term is bounded by Lemma \ref{intofthemartcondcasneg}.  Collecting these bounds, we obtain
\begin{align}
	J_0 \geq\; & \bb E_{x} \left( W\left( X_{\pent{n^{1-\ee}}},y+S_{\pent{n^{1-\ee}}} \right) \,;\, \tau_y > \pent{n^{1-\ee}} \right) \nonumber \\
	&- c_{p,\ee} \bb E_{x} \left( \sqrt{n} + \abs{X_{m_\ee}} \,;\, \tau_y > m_\ee \,,\, v_n > m_\ee \right) -\frac{c_{p,\ee}}{n^{\ee}} \left( 1+y+\abs{x}^p \right).
\label{Bound-J0-005}
\end{align}
Coupling the H\"older inequality with Lemma \ref{MCM} and Lemma \ref{concentnu},
we find that the second term in the r.h.s.\ of (\ref{Bound-J0-005})  does not exceed
\[
c_{p,\ee} \left( \sqrt{n} + \bb E_{x}^{1/p} \left( \abs{X_{m_\ee}}^p \right) \right) \bb P_x^{1/q} \left( v_n > \frac{n^{1-\ee}}{c_\ee} \right) \leq
c_{p,\ee} \left( \sqrt{n} + \abs{x} \right) \frac{c_{p,\ee} \left( 1 + \abs{x} \right)^{p-1}}{ n^{\frac{p-1}{2}-(p-1)\ee} }.
\]
Implementing this into (\ref{Bound-J0-005}),
	\[
	J_0 \geq \bb E_{x} \left( W\left( X_{\pent{n^{1-\ee}}},y+S_{\pent{n^{1-\ee}}} \right) \,;\, \tau_y > \pent{n^{1-\ee}} \right) -\frac{c_{p,\ee}}{n^{\ee}} \left( 1+y+\abs{x}^p \right).
\]
Since $\left(W\left(X_n,y+S_n\right) \mathbbm{1}_{\{\tau_y>n\}} \right)_{n\geq 1}$ is a supermartingale, Lemma \ref{lemanalyse2} implies that
\[
J_0 \geq \bb E_{x} \left( W\left( X_{n_f},y+S_{n_f} \right) \,;\, \tau_y > n_f \right) -\frac{c_{p,\ee}}{n_f^{\ee}} \left( 1+y+\abs{x}^p \right).
\]
Using the lower bound of the claim 3 of Proposition \ref{InvfunctpropW} and Lemma \ref{majmartx}, we deduce that
	\[
	\bb E_x \left( W\left(X_n,y+S_n\right) \,;\, \tau_y>n \right) \geq y \bb P_{x} \left( \tau_y > n_f \right) - y \frac{c_{p,\ee}}{n_f^\ee} 
	- c_{p,\ee} \sqrt{n_f} - c_{p,\ee} \abs{x}^p.
\]
Now, when $y\to +\infty$, one can see that $\bb P_{x} \left( \tau_y > n_f \right) \to 1$: more precisely,
	\[
	\bb P_x \left( \tau_y > n_f \right) \geq \bb P_x \left( \underset{1 \leq k \leq n_f}{\max} \abs{X_k} < \frac{y}{n_f} \right) \geq 1 - c \frac{n_f^2 \left( 1+\abs{x} \right)}{y}.
\]
Finally,
\begin{align*}
	\bb E_x &\left( W\left(X_n,y+S_n\right) \,;\, \tau_y>n \right) \geq y \left( 1 - \frac{c_{p,\ee}}{n_f^\ee} \right)  - c_{p,\ee} n_f^2 \left( 1+\abs{x}^p \right).
\end{align*}
\end{proof}

Under Condition \ref{PosdeX1-2} we use Lemma \ref{lowerboundcasneg} to prove that $V$ is positive on $\bb R \times \bb R_+^*$.
\begin{lemma}
\label{Invfunctpropcasnegpart2}\ 
\begin{enumerate}
\item For any $\delta > 0$, $p \in (2,\alpha)$, $x \in \bb R$, $y>0$,
	\[
	V(x,y) \geq (1-\delta)y - c_{p,\delta} \left( 1+\abs{x}^p \right).
\]
\item For any $x \in \bb R$,
	\[
	\underset{y \to +\infty}{\lim} \frac{V(x,y)}{y} = 1.
\]
\item The function $V$ is positive on $\bb R \times \bb R_+^*$.
\end{enumerate}
\end{lemma}

\begin{proof}
\textit{Claim 1.} Using the claim 1 of Lemma \ref{MTyprop} and the claims 3 and 5 of Proposition \ref{InvfunctpropW}, with $z=y+\rho x$, we write
\begin{align*}
	\bb E_x &\left( z + M_n \,;\, \tau_y >n \right)\\
	&\geq \bb E_x \left( z + M_n \,;\, T_y >n \right) - \bb E_x \left( W(X_n,y+S_n) \,;\, T_y >n \,,\, \tau_y \leq n \right) \\
	&= \bb E_x \left( z + M_n \,;\, T_y >n \right) - W(x,y) + \bb E_x \left( W(X_n,y+S_n) \,;\, \tau_y > n \right).
\end{align*}
Using Lemma \ref{lowerboundcasneg}, the claim 1 of Proposition \ref{InvfunctpropW} and the claim 1 of Proposition \ref{Invfunctpropcasnegpart1}, we obtain
	\[
	V(x,y) \geq y \left( 1 - \frac{c_{p,\ee}}{n_f^\ee} \right)  - c_{p,\ee} n_f^2 \left( 1+\abs{x}^p \right).
\]
Taking $n_f$ large enough, the claim 1 is proved.

\textit{Claim 2.} Taking the limit as $y \to +\infty$ and as $\delta \to 0$ in the claim 1, we obtain first that $\underset{y \to +\infty}{\rm liminf}V(x,y)/y \geq 1$. By the claim 3 of Proposition \ref{Invfunctpropcasnegpart1}, 
we obtain also that $\underset{y \to +\infty}{\rm limsup} V(x,y)/y \leq 1$.

\textit{Claim 3.} Fix $x\in \bb R$, $y>0$ and $\delta_0>0$. By Condition \ref{PosdeX1-2}, or Condition \ref{CSPosdeX1-2} 
(see Section \ref{ComplementCond}), 
there exists $p_0 \in (2,\alpha)$ such that for any $c>0$ there exists $n_0 \geq 1$ such that $\bb P_x \left( \left( X_{n_0}, y+S_{n_0} \right) \in K_{p_0, c} \,,\, \tau_y > n_0 \right) > 0$. Thus, using the claim 4 of Proposition \ref{Invfunctpropcasnegpart1}, 
	\[
	V(x,y) \geq \bb E_x \left( V(X_{n_0},y+S_{n_0}) \,;\, \left( X_{n_0}, y+S_{n_0} \right) \in K_{p_0, c} \,,\, \tau_y > n_0 \right).
\]
Using the claim 1 with $p=p_0$ and $\delta =1/2$ and choosing the constant $c=2c_{p_0,\delta}+2\delta_0$, there exists $n_0$ such that
	\[
	V(x,y) \geq \delta_0 \bb P_x \left( \left( X_{n_0}, y+S_{n_0} \right) \in K_{p_0, c} \,,\, \tau_y > n_0 \right) > 0.
\]
\end{proof}

\section{Asymptotic for the exit time}
\label{As for Exit Time}

The aim of this section is to prove Theorem \ref{AsExTi}. The asymptotic for the exit time of the Markov walk $(y+S_n)_{n\geq 0}$ will be deduced from the asymptotic of the exit time for the Brownian motion in Corollary \ref{exittimeforB} using the functional approximation in Proposition \ref{majdeA_k}.

\subsection{Auxiliary statements}

We start by proving an analogue of Corollaries \ref{intofrandwalk} and \ref{intofrandwalkcasneg}, 
where $n$ is replaced by the stopping time $\nu_n$.
\begin{lemma}
\label{lemsurE1}
For any $p \in (2,\alpha)$, there exists $\ee_0 >0$ such that for any $\ee \in (0, \ee_0]$, $x \in \bb R$, $y>0$ and $n \geq 1$,
	\[
	E_1 = \bb E_x \left( y+S_{\nu_n} \,;\, \tau_y > \nu_n \,,\, \nu_n \leq \pent{n^{1-\ee}} \right) \leq c_{p,\ee} (1+ y + \abs{x} )(1+\abs{x})^{p-1}.
\]
\end{lemma}

\begin{proof} When $\tau_y > \nu_n>1$, we note that
\begin{equation}
\label{Xnunstrictpos}
0< X_{\nu_n} < y +S_{\nu_n}.
\end{equation}
Therefore, using the martingale representation (\ref{MfSetX}), we have $y+S_{\nu_n} \leq z+M_{\nu_n} + \max(0,-\rho) X_{\nu_n},$ 
with $z=y +\rho x,$
 and so
\[ 0<y+S_{\nu_n} \leq \max\left( 1,1-\bb E(a) \right) \left( z+M_{\nu_n} \right) = c \left( z+M_{\nu_n} \right). \]
Consequently,
\begin{align}
	E_1 \leq\; & c\left(1+y+\abs{x}\right) + c\bb E_x \left( z + M_{\nu_n} \,;\, \tau_y > \nu_n \,,\, 1<\nu_n \leq \pent{n^{1-\ee}} \right) \nonumber\\
	\leq\; & c\left(1+y+\abs{x}\right) + \underbrace{c\bb E_x \left( z + M_{\nu_n} \,;\, \tau_y > \nu_n \,,\, \nu_n \leq \pent{n^{1-\ee}} \right)}_{E_1'}. \label{E1'001}
\end{align}
Now, denoting $\nu_n \wedge \pent{n^{1-\ee}}= \min (\nu_n, \pent{n^{1-\ee}})$, we write
\begin{align*}
	E_1'	= \; &c\bb E_x \left( z + M_{\nu_n \wedge \pent{n^{1-\ee}}} \right) - c\bb E_x \left( z + M_{\nu_n \wedge \pent{n^{1-\ee}}} \,;\, \tau_y \leq  \nu_n \wedge \pent{n^{1-\ee}} \right) \\
	&- c\bb E_x \left( z + M_{\pent{n^{1-\ee}}} \,;\, \tau_y > \pent{n^{1-\ee}} \,,\, \nu_n > \pent{n^{1-\ee}} \right).
\end{align*} 
Since $\left( M_n \right)_{n\geq 0}$ is a centred martingale, 
using Lemma \ref{intdeMtau}  when $\bb E (a) \geq 0$ and Lemma \ref{intdeMtaucasneg} when $\bb E (a) < 0$, Lemmas \ref{majmartx}, \ref{concentnu} and H\"older inequality, we obtain
	\[
	E_1' \leq c_{p,\ee} (1+ y + \abs{x} )(1+\abs{x})^{p-1}.
\]
Implementing this into (\ref{E1'001}), it concludes the proof.
\end{proof}

Now, we can prove an upper bound of order $1/n^{1/2-c\ee}$ of the probability of survival $\bb P_x \left( \tau_y > n \right)$.
\begin{lemma}
\label{Sommedestempsdesurvie}
For any $p \in (2,\alpha)$, there exists $\ee_0>0$ such that for any $\ee \in (0, \ee_0]$, $x\in \bb R$, $y>0$ and $n\geq 1$,
	\[
	\bb P_x \left( \tau_y > n \right) \leq c_{p,\ee} \frac{(1+ y + \abs{x})(1+\abs{x})^{p-1}}{n^{1/2- \ee}}.
\]
Moreover, summing these bounds, we have
	\[
	\sum_{k=1}^{\pent{n^{1-\ee}}} \bb P_x \left( \tau_y > k \right) \leq c_{p,\ee} (1+ y + \abs{x})(1+\abs{x})^{p-1} n^{1/2+ \ee}.
\]
\end{lemma}

\begin{proof}
Taking $k=\pent{n^{\frac{1}{1-\ee}}}$, we write
	\[
	\bb P_x \left( \tau_y > n \right) \leq \bb E_x \left( \frac{y+S_{\nu_k}}{k^{1/2-\ee}} \,;\, \tau_y > \nu_k \,,\, \nu_k \leq \pent{k^{1-\ee}} \right) + \bb P_x \left( \nu_k > \frac{k^{1-\ee}}{c_\ee} \right).
\]
Using Lemma \ref{lemsurE1} and Lemma \ref{concentnu}, the claim follows.
\end{proof}

Before to proceed with the proof of Theorem \ref{AsExTi}, we need two additional technical lemmas. Recall the notation 
$\nu_n^{\ee/6} = \nu_n + \pent{n^{\ee/6}}$.
\begin{lemma}
\label{lemsurE2}
There exists $\ee_0 >0$ such that for any $\ee \in (0, \ee_0]$, $x \in \bb R$ and $y>0$,
	\[
	E_2 = \bb E_x \left( y+S_{\nu_n^{\ee/6}} \,;\, \tau_y > \nu_n^{\ee/6} \,,\, \nu_n^{\ee/6} \leq \pent{n^{1-\ee}} \right) \underset{n\to +\infty}{\longrightarrow} V(x,y).
\]
\end{lemma}

\begin{proof} Using the martingale approximation (\ref{MfSetX}),
\begin{align}
	E_2 =\;&\underbrace{- \rho \bb E_x \left( X_{\nu_n^{\ee/6}} \,;\, \tau_y > \nu_n^{\ee/6} \,,\, \nu_n^{\ee/6} \leq \pent{n^{1-\ee}} \right)}_{=: E_{21}} \nonumber\\
	\label{decompositiondeE2}
	& \qquad +\underbrace{\bb E_x \left( z+M_{\nu_n^{\ee/6}} \,;\, \tau_y > \nu_n^{\ee/6} \,,\, \nu_n^{\ee/6} \leq \pent{n^{1-\ee}} \right)}_{=: E_{22}}.
\end{align}

\textit{Bound of $E_{21}$.} By the Markov property, Lemma \ref{MCM} and the fact that $(y+S_{\nu_n})/n^{1/2-\ee} > 1$,
\begin{align*}
	\abs{E_{21}} \leq\; &c \bb E_x \left( 1+ \e^{-c n^{\ee/6}} \abs{X_{\nu_n}} \,;\, \tau_y > \nu_n  \,,\, \nu_n \leq \pent{n^{1-\ee}} \right)\\
	\leq \; &\frac{c}{n^{1/2-\ee}} E_1+ \e^{-c n^{\ee/6}} \sum_{k=1}^{\pent{n^{1-\ee}}} \bb E_x \left( \abs{X_{k}} \right).
\end{align*}
By Lemma \ref{lemsurE1}, we obtain
\begin{equation}
	\label{BoundE22}
	\abs{E_{21}} \leq c_{p,\ee} \frac{(1+ y + \abs{x} )(1+\abs{x})^{p-1}}{n^{1/2-\ee}}.
\end{equation}

\textit{Bound of $E_{22}$.} 
We proceed in the same way as for bounding  $E_1'$ defined in (\ref{E1'001}):
\begin{align*}
	E_{22} = z - &\bb E_x \left( z+M_{\tau_y} \,;\, \tau_y \leq \nu_n^{\ee/6} \wedge \pent{n^{1-\ee}} \right) \\
	- &\bb E_x \left( z+M_{\nu_n^{\ee/6} \wedge \pent{n^{1-\ee}}} \,;\, \tau_y > \nu_n^{\ee/6} \wedge \pent{n^{1-\ee}} \,,\, \nu_n^{\ee/6} > \pent{n^{1-\ee}} \right).
\end{align*}
By the H\"older inequality, Lemma \ref{majmartx} and Lemma \ref{concentnu},
\begin{equation}
	\label{majdeE22}
	E_{22} \leq z - \bb E_x \left( z+M_{\tau_y} \,;\, \tau_y \leq \nu_n^{\ee/6} \wedge \pent{n^{1-\ee}} \right) + c_{p,\ee} \frac{(1+ y + \abs{x} )(1+\abs{x})^{p-1}}{n^{\frac{p-2}{2}-c_p\ee}}.
\end{equation}
Since $\nu_n^{\ee/6} \geq \pent{n^{\ee/6}} \to +\infty$ as $n\to +\infty$ and $M_{\tau_y}$ is integrable (according to Lemmas \ref{intdeMtau} and \ref{intdeMtaucasneg}), by the Lebesgue dominated convergence we deduce  that
	\[
	\underset{n \to +\infty}{\lim} E_{22} = -\bb E_x \left(M_{\tau_y} \right) = V(x,y).
\]
Coupling this with equations (\ref{decompositiondeE2}) and (\ref{BoundE22}), we conclude that $E_2 \underset{n\to +\infty}{\longrightarrow} V(x,y)$.
\end{proof}

\begin{lemma}
\label{lemsurE3}
There exists $\ee_0 >0$ such that for any $\ee \in (0, \ee_0]$, $x \in \bb R$ and $y>0$,
\begin{align*}
	E_3 &= \bb E_x \left( y+S_{\nu_n^{\ee/6}} \,;\, y+S_{\nu_n^{\ee/6}} > n^{1/2-\ee/6} \,,\, \tau_y > \nu_n^{\ee/6} \,,\, \nu_n^{\ee/6} \leq \pent{n^{1-\ee}} \right) \underset{n\to +\infty}{\longrightarrow} 0.
\end{align*}
\end{lemma}

\begin{proof} The first step of the proof consists in proving that we can replace the time $\nu_n^{\ee/6}$ in the definition of $E_3$ by the time $\nu_n$. More precisely, we shall prove that the following bound holds true: 
\begin{align}
	E_3 \leq \; & c n^{\ee/6} \underbrace{\bb E_x \left( y+S_{\nu_n} \,;\, y+S_{\nu_n} > n^{1/2-\ee/2} \,,\,  \tau_y > \nu_n \,,\, \nu_n \leq \pent{n^{1-\ee}} \right)}_{=:E_{31}} \nonumber\\
	    & +  c_{p,\ee}\frac{(1+ y + \abs{x} )(1+\abs{x})^{p-1}}{n^{\ee/6}}.
	    \label{conclusionpremierpas}
\end{align}
To this end, we bound $E_3$ as follows:
\begin{align}
	E_3	\leq & E_{31} + \underbrace{\bb E_x \left( \abs{S_{\nu_n^{\ee/6}}-S_{\nu_n}} \,;\, y+S_{\nu_n} > n^{1/2-\ee/2} \,;\, \tau_y > \nu_n \,,\, \nu_n \leq \pent{n^{1-\ee}} \right)}_{=:E_{32}} \nonumber\\
	+&\bb E_x \left( y+S_{\nu_n} \,;\, y+S_{\nu_n} \leq n^{1/2-\ee/2} \,,\, y+S_{\nu_n^{\ee/6}} > n^{1/2-\ee/6} \,,\, \right. \nonumber\\
	\label{decompositiondeE3}
	&\underbrace{\hspace{8cm} \left. \tau_y > \nu_n \,,\, \nu_n \leq \pent{n^{1-\ee}} \right)}_{=:E_{33}} \\
	+&\bb E_x \left( \abs{S_{\nu_n^{\ee/6}}-S_{\nu_n}} \,;\, y+S_{\nu_n} \leq n^{1/2-\ee/2} \,,\, y+S_{\nu_n^{\ee/6}} > n^{1/2-\ee/6} \,,\, \right. \nonumber\\
	&\underbrace{\hspace{8cm} \left. \tau_y > \nu_n \,,\, \nu_n \leq \pent{n^{1-\ee}} \right)}_{=:E_{34}}. \nonumber
\end{align}

\textit{Bound of $E_{32}$.} By the Markov property and Lemma \ref{MCM},
\begin{align*}
	E_{32} &\leq \int_{\bb R \times \bb R_+^*} \bb E_{x'} \left( \abs{S_{\pent{n^{\ee/6}}}} \right) \bb P_x \left( X_{\nu_n} \in \dd x' \,,\, y+S_{\nu_n} \in \dd y' \,,\, \right.\\
	&\hspace{6cm} \left. y+S_{\nu_n} > n^{1/2-\ee/2} \,,\, \tau_y > \nu_n \,,\, \nu_n \leq \pent{n^{1-\ee}} \right) \\
	&\leq \bb E_x \left( c n^{\ee/6} \left( 1 + \abs{X_{\nu_n}} \right) \,;\, y+S_{\nu_n} > n^{1/2-\ee/2} \,,\, \tau_y > \nu_n \,,\, \nu_n \leq \pent{n^{1-\ee}} \right).
\end{align*}
If $\tau_y > \nu_n >1 $, by (\ref{Xnunstrictpos}), we have $\abs{X_{\nu_n}} = X_{\nu_n} < y+S_{\nu_n}$. 
Using this bound when $\nu_n>1$ and the Markov inequality when $\nu_n=1$,
\begin{align}
	E_{32} \leq\; & \bb E_x \left( c n^{\ee/6} \left( 1 + \abs{X_1} \right) \,;\, y+X_1 > n^{1/2-\ee/2} \,,\, \nu_n =1 \right) + c n^{\ee/6} E_{31} \nonumber\\
	\leq\; & c \frac{(1+y+\abs{x})(1+\abs{x})}{n^{1/2-c\ee}} + c n^{\ee/6} E_{31}. \label{BoundE32}
\end{align}

\textit{Bound of $E_{33}$.} By the Markov property,
\begin{align*}
	E_{33} &\leq \int_{\bb R \times \bb R_+^*} y' \bb P_{x'} \left( y'+S_{\pent{n^{\ee/6}}} > n^{1/2-\ee/6} \right) \bb P_x \left( X_{\nu_n} \in \dd x' \,,\, y+S_{\nu_n} \in \dd y'  \,,\, \right. \\
	&\hspace{6cm} \left. y+S_{\nu_n} \leq n^{1/2-\ee/2} \,,\, \tau_y > \nu_n \,,\, \nu_n \leq \pent{n^{1-\ee}} \right).
\end{align*}
When $y' \leq n^{1/2-\ee/2}$,  by the Markov inequality, we have,
	\[
	\bb P_{x'} \left( y'+S_{\pent{n^{\ee/6}}} > n^{1/2-\ee/6} \right) \leq \bb P_{x'} \left( \abs{S_{\pent{n^{\ee/6}}}} > \frac{n^{1/2-\ee/6}}{c_{\ee}} \right) \leq \frac{c_{\ee}n^{\ee/6} \left(1 + \abs{x'} \right)}{n^{1/2-\ee/6}}.
\]
On the event $\{ y+S_{\nu_n} \leq n^{1/2-\ee/2} \,,\, \tau_y > \nu_n \}$, we obviously have $x'=X_{\nu_n} \leq n^{1/2-\ee/2}$. 
From these bounds, using the positivity of $X_{\nu_n}$ for $\nu_n >1$, see (\ref{Xnunstrictpos}), we obtain
	\[
	E_{33} \leq \bb E_x \left( \left( y+S_1 \right) \frac{c_{\ee}\left( 1 + \abs{X_1} \right)}{n^{1/2-\ee/3}} \,;\, \nu_n = 1 \right) + \frac{c_{\ee}}{n^{\ee/2-\ee/3}} E_1.
\]
By Lemma \ref{lemsurE1}, we obtain
\begin{equation}
	\label{BoundE33}
	E_{33} \leq c_{p,\ee}\frac{(1+y+\abs{x})(1+\abs{x})^{p-1}}{n^{\ee/6}}.
\end{equation}

\textit{Bound of $E_{34}$.} Again, by the Markov property,
\begin{align*}
	E_{34} \leq\;& \int_{\bb R \times \bb R_+^*} \bb E_{x'} \left( \abs{S_{\pent{n^{\ee/6}}}} \,;\, y'+S_{\pent{n^{\ee/6}}} > n^{1/2-\ee/6} \right) \bb P_x \left( X_{\nu_n} \in \dd x' \,,\, \right.\\
	&\hspace{2cm} \left. y+S_{\nu_n} \in \dd y' \,,\, y+S_{\nu_n} \leq n^{1/2-\ee/2} \,,\, \tau_y > \nu_n \,,\, \nu_n \leq \pent{n^{1-\ee}} \right).
\end{align*}
When $y' \leq n^{1/2-\ee/2}$, we have
	\[
	\bb E_{x'} \left( \abs{S_{\pent{n^{\ee/6}}}} \,;\, y'+S_{\pent{n^{\ee/6}}} > n^{1/2-\ee/6} \right) \leq \bb E_{x'} \left( \frac{c_{\ee}^{p-1}\abs{S_{\pent{n^{\ee/6}}}}^p}{n^{\frac{p-1}{2}-(p-1)\ee/6}} \right) \leq c_{p,\ee} \frac{\left(1+\abs{x'}\right)^p}{n^{\frac{p-1}{2}-c_p\ee}}.
\]
Then, using Lemma \ref{MCM},
\begin{align*}
	E_{34} &\leq \frac{c_{p,\ee}}{n^{\frac{p-1}{2}-c_p \ee}} + \frac{c_{p,\ee}}{n^{\frac{p-1}{2}-c_p \ee}} \sum_{k=1}^{\pent{n^\ee}} \bb E_x \left(  \abs{X_k}^p \right) + \frac{c_{p,\ee}}{n^{\frac{p-1}{2}-c_p \ee}} \sum_{k=\pent{n^\ee}+1}^{\pent{n^{1-\ee}}} \bb E_x \left(  \abs{X_k}^p \,;\, \tau_y > k \right) \\
	&\leq \frac{c_{p,\ee}\left(  1+\abs{x}^p \right)}{n^{\frac{p-1}{2}-c_p \ee}} + \frac{c_{p,\ee}}{n^{\frac{p-1}{2}-c_p \ee}} \sum_{k=1}^{\pent{n^{1-\ee}}-\pent{n^{\ee}}} \bb E_x \left( 1 + \e^{-c_{p} n^{\ee}} \abs{X_{k}}^p \,;\, \tau_y > k \right) \\
	&\leq \frac{c_{p,\ee}\left(  1+\abs{x}^p \right)}{n^{\frac{p-1}{2}-c_p \ee}} + \e^{-c_{p,\ee} n^{\ee}} \left(  1+\abs{x}^p \right) + \frac{c_{p,\ee}}{n^{\frac{p-1}{2}-c_p \ee}} \sum_{k=1}^{\pent{n^{1-\ee}}} \bb P_x \left( \tau_y > k \right).
\end{align*}
Using the second bound in Lemma \ref{Sommedestempsdesurvie}, and taking $\ee>0$ small enough, we obtain
\begin{equation}
	\label{BoundE34}
	E_{34} \leq c_{p,\ee}\frac{(1+ y + \abs{x} )(1+\abs{x})^{p-1}}{n^{\frac{p-2}{2}-c_p \ee}} \underset{n\to +\infty}{\longrightarrow} 0.
\end{equation}
Inserting (\ref{BoundE32}), (\ref{BoundE33}) and (\ref{BoundE34}) in (\ref{decompositiondeE3}), we conclude the proof of 
\eqref{conclusionpremierpas}.

\textit{Bound of $c n^{\ee/6} E_{31}$.} Note that, when $\nu_n>1$ and $y+S_{\nu_n} > n^{1/2-\ee/2}$, we have	$X_{\nu_n} = y+S_{\nu_n} - y+S_{\nu_n-1} > n^{1/2-\ee/2} - n^{1/2-\ee} \geq \frac{n^{1/2-\ee/2}}{c_{\ee}}$. Consequently,
\begin{align}
	c n^{\ee/6} E_{31} \leq\;& \underbrace{c n^{\ee/6} \bb E_x \left( y+S_{\nu_n} \,;\,  \nu_n \leq \pent{n^{\ee}} \right)}_{=: E_{35}} \nonumber\\
	&+\underbrace{c n^{\ee/6} \bb E_x \left( y+S_{\nu_n} \,;\, X_{\nu_n} > \frac{n^{1/2-\ee/2}}{c_{\ee}} \,,\, \tau_y > \nu_n \,,\, \pent{n^{\ee}} < \nu_n \leq \pent{n^{1-\ee}} \right)}_{=: E_{36}}.
	\label{decomposiotiondeE31'}
\end{align}

\textit{Bound of $E_{35}$.}  Using the definition of $\nu_n$, the Markov inequality and Lemma \ref{MCM},
\begin{align}
	E_{35} &\leq c n^{\ee/6} \bb E_x \left( \underset{k \leq \pent{n^{\ee}}}{\max} \abs{y+S_k} \,;\, \underset{k\leq \pent{n^{\ee}}}{\max} \abs{y+S_k} > n^{1/2-\ee} \right) \nonumber\\
	\label{BoundE35}
	&\leq \frac{c_p \left(1+y+\abs{x}\right)^2}{n^{1/2-c_p\ee}}.
\end{align}

\textit{Bound of $E_{36}$.}
The idea is based on the observation that, according to the first bound in Lemma \ref{MCM}, the random variables $y+S_{\nu_n-\pent{n^{\ee}}}$ and $X_{\nu_n}$ are "almost" independent. In this line, summing over the values of $\nu_n$ and bounding the indicators $\mathbbm 1_{\{ \nu_n = k\}}$ by 1, we write
\begin{align*}
	E_{36} \leq &c n^{\ee/6} \sum_{k=\pent{n^{\ee}} +1 }^{\pent{n^{1-\ee}}} \bb E_x \left( y+S_{k-\pent{n^{\ee}}} \,;\, X_k > \frac{n^{1/2-\ee/2}}{c_{\ee}} \,,\, \tau_y > k \right) \\
	&+c n^{\ee/6} \sum_{k=\pent{n^{\ee}} +1}^{\pent{n^{1-\ee}}} \bb E_x \left( \abs{S_k - S_{k-\pent{n^{\ee}}}} \,;\, X_k > \frac{n^{1/2-\ee/2}}{c_{\ee}} \,,\, \tau_y > k \right).
\end{align*}
By the Markov property,
\begin{align}
E_{36} \leq\; &c n^{\ee/6} \sum_{k=\pent{n^{\ee}}+1}^{\pent{n^{1-\ee}}} \int_{\bb R \times \bb R_+^*} y' \bb P_{x'} \left( X_{\pent{n^{\ee}}} > \frac{n^{1/2-\ee/2}}{c_{\ee}} \right) \nonumber\\
& \hspace{3cm} \times \bb P_x \left( X_{k-\pent{n^{\ee}}} \in \dd x' \,,\, y+S_{k-\pent{n^{\ee}}} \in \dd y' \,,\, \tau_y > k-\pent{n^{\ee}} \right) \nonumber \\
	&+c n^{\ee/6} \sum_{k=\pent{n^{\ee}}+1}^{\pent{n^{1-\ee}}} \bb E_x \left( n^{\ee} \underset{k-\pent{n^{\ee}} \leq i \leq k}{\max} \abs{X_i} \,;\, X_k > \frac{n^{1/2-\ee/2}}{c_{\ee}} \,,\, \tau_y > k \right). \label{EEE_001}
\end{align}
Recall that, under $\{ X_0=x' \}$ by (\ref{defdeXn0}),  $X_{\pent{n^{\ee}}}= \prod_{i=1}^{\pent{n^\ee}} a_i x' + X_{\pent{n^{\ee}}}^0$. 
Then, since $a_i$'s are independent and identically distributed, by claim 1 of Condition \ref{Mom001} and Lemma \ref{MCM}, 
\begin{align}
	\bb P_{x'} \left( X_{\pent{n^{\ee}}} > \frac{n^{1/2-\ee/2}}{c_{\ee}} \right) &\leq \bb P \left( \prod_{i=1}^{\pent{n^\ee}} a_i x' > \frac{n^{1/2-\ee/2}}{2c_{\ee}} \right) + \bb P \left( \abs{{X_{\pent{n^{\ee}}}^0}} > \frac{n^{1/2-\ee/2}}{2c_{\ee}} \right) \nonumber \\
	&\leq \e^{-c_{\ee}n^{\ee}}\abs{x'} + \frac{c_{p,\ee}}{n^{\frac{p}{2}-c_p\ee}}.
	\label{EEE_002}
\end{align}
Inserting (\ref{EEE_002}) into (\ref{EEE_001}) and using Cauchy-Schwartz inequality, by Corollaries \ref{intofrandwalk} and \ref{intofrandwalkcasneg},
\begin{align*}
	E_{36} \leq \; &\sum_{j=1}^{\pent{n^{1-\ee}}}
\left(\e^{-c_p n^{\ee}}\bb E_x^{1/2} \left( \abs{y+S_j}^2 \right)\bb E_x^{1/2} \left( \abs{X_j}^2 \right) + \frac{c_{p,\ee}}{n^{\frac{p}{2}-c_p\ee}} (1+ y + \abs{x} )(1+\abs{x})^{p-1}\right) \\
	&+c n^{\ee+\ee/6} \sum_{k=\pent{n^{\ee}}+1}^{\pent{n^{1-\ee}}} \bb E_x \left( \frac{\underset{k-\pent{n^{\ee}}\leq i \leq k}{\max} \abs{X_i}^p}{n^{\frac{p-1}{2}-c_p\ee}} \,;\, \tau_y > k-\pent{n^\ee} \right).
\end{align*}
Using the decomposition (\ref{MfSetX}) and Lemmas \ref{MCM} and \ref{majmartx}
\begin{align*}
	E_{36} \leq \; &c_{p,\ee} \frac{(1+ y + \abs{x} )(1+\abs{x})^{p-1}}{n^{\frac{p-2}{2}-c_p\ee}} \\
	&+ \frac{c_p}{n^{\frac{p-1}{2}-c_p\ee}} \sum_{k=\pent{n^{\ee}}+1}^{\pent{n^{1-\ee}}} \bb E_x \left( n^{\ee} \left( 1+ \abs{X_{k-\pent{n^\ee}}}^p \right) \,;\, \tau_y > k-\pent{n^\ee} \right).
\end{align*}
Re-indexing $j=k-\pent{n^\ee}$, after some elementary transformations, we get
\begin{align*}
	E_{36} \leq \; 	& c_{p,\ee} \frac{(1+ y + \abs{x} )(1+\abs{x})^{p-1}}{n^{\frac{p-2}{2}-c_p\ee}} + \frac{c_p}{n^{\frac{p-1}{2}-c_p\ee}} \sum_{j=1}^{\pent{n^{1-\ee}}} \bb P_x \left( \tau_y > j \right) \\
	&+ \frac{c_p}{n^{\frac{p-1}{2}-c_p\ee}} \sum_{j=1}^{\pent{n^{\ee}}} \bb E_x \left( \abs{X_j}^p \right) + \frac{c_p}{n^{\frac{p-1}{2}-c_p\ee}} \sum_{j=\pent{n^{\ee}}+1}^{\pent{n^{1-\ee}}} \bb E_x \left( \abs{X_j}^p \,;\, \tau_y > j-\pent{n^\ee} \right).
\end{align*}
Again using the Markov property, Lemma \ref{MCM} and Lemma \ref{Sommedestempsdesurvie}, we have
\begin{align*}
	E_{36} \leq \; &c_{p,\ee} \frac{(1+ y + \abs{x})(1+\abs{x})^{p-1}}{n^{\frac{p-2}{2}-c_p\ee}} + \frac{c_p}{n^{\frac{p-1}{2}-c_p\ee}} \sum_{j=1}^{\pent{n^{1-\ee}}} \bb P_x \left( \tau_y > j \right) \\
	&+  e^{-c_p n^\ee} \sum_{j=1}^{\pent{n^{1-\ee}}} \bb E_x \left( \abs{X_j}^p \,;\, \tau_y > j \right) \\
	\leq \; &c_{p,\ee} \frac{(1+ y + \abs{x})(1+\abs{x})^{p-1}}{n^{\frac{p-2}{2}-c_p\ee}}.
\end{align*}
Inserting this bound and (\ref{BoundE35}) into (\ref{decomposiotiondeE31'}), we obtain 
	\[
	c n^{\ee/6} E_{31} \leq \frac{ c_{p,\ee} \left( 1+y+\abs{x} \right)^p }{n^{\frac{p-2}{2}-c_p\ee}}.
\]
Together with (\ref{conclusionpremierpas}), this  bound implies that
\begin{equation}
	\label{belmajdeE3}
	E_3 \leq \frac{c_p \left( 1+y+\abs{x} \right)^p}{n^{\ee/6}} \underset{n\to +\infty}{\longrightarrow} 0.
\end{equation}
\end{proof}

\subsection{Proof of the claim 2 of theorem \ref{AsExTi}}
\label{claim2thasexti}

Introducing the stopping time $\nu_n^{\ee/6} = \nu_n + \pent{n^{\ee/6}}$, we have
\begin{equation}
	\bb P_x \left( \tau_y > n \right) = \bb P_x \left( \tau_y > n \,,\, \nu_n^{\ee/6} \leq \pent{n^{1-\ee}} \right) + \bb P_x \left( \tau_y > n \,,\, \nu_n^{\ee/6} > \pent{n^{1-\ee}} \right).
\label{DecompJ0-001}
\end{equation}
We bound the second term by Lemma \ref{concentnu}: for $2< p < \alpha$,
\begin{equation}
	\bb P_x \left( \tau_y > n \,,\, \nu_n^{\ee/6} > \pent{n^{1-\ee}} \right) \leq \bb P_x \left( \nu_n > \frac{n^{1-\ee}}{c_{\ee}} \right) \leq c_{p,\ee} \frac{(1+\abs{x})^p}{n^{p/2-c_p \ee}} = o\left(\frac{1}{\sqrt{n}}\right).
\label{BoundJ0-002}
\end{equation}
To bound the first term, we introduce more notations. Let $(B_t)_{t \geq 0}$ be the Brownian motion from Proposition \ref{majdeA_k}, $A_k$ be the event $A_k=\{ \underset{0 \leq t \leq 1}{\max} \abs{ S_{\pent{tk}} - \sigma B_{tk} } \leq k^{1/2-2\ee} \}$ where $\sigma$ is defined by (\ref{defdesigma}), and $\overline{A}_k$ be its complement. Using the Markov property, we have
\begin{align}
	\bb P_x \left( \tau_y > n \,,\, \nu_n^{\ee/6} \leq \pent{n^{1-\ee}} \right) =\;& \sum_{k=1}^{\pent{n^{1-\ee}}} \int_{\bb R \times \bb R_+^*} \bb P_{x'} \left(\tau_{y'} > n-k \,,\, \overline{A}_{n-k} \right) \bb P_x \left( X_k \in \dd x' \,,\, \right. \nonumber\\
	&\underbrace{\hspace{3cm} \left. y+S_k \in \dd y' \,,\, \tau_y > k \,,\, \nu_n^{\ee/6} =k \right)}_{=:J_1} \nonumber\\
	\label{decompositiondeT}
	+&\sum_{k=1}^{\pent{n^{1-\ee}}} \int_{\bb R \times \bb R_+^*} \bb P_{x'} \left(\tau_{y'} > n-k \,,\, A_{n-k} \right) \bb P_x \left( X_k \in \dd x' \,,\, \right.\\
	&\underbrace{\hspace{3cm} \left. y+S_k \in \dd y' \,,\, \tau_y > k \,,\, \nu_n^{\ee/6} =k \right)}_{=: J_2}. \nonumber
\end{align}

\textit{Bound of $J_1$.} Taking into account that $n-k \geq \frac{n}{c_{\ee}}$ for any $k\leq \pent{n^{1-\ee}}$, by Proposition \ref{majdeA_k} with $\ee$ small enough, we find 
	\[
	\bb P_{x'} \left(\tau_{y'} > n-k \,,\, \overline{A}_{n-k} \right) \leq \bb P_{x'} \left( \overline{A}_{n-k} \right) \leq c_{p,\ee} (1+\abs{x'} )^p n^{-2\ee}.
\]
By the Markov property and the first bound in Lemma \ref{MCM},
	\[
	J_1 \leq \bb E_x \left( \e^{-c_{p,\ee} n^{\ee/6}} \abs{X_{\nu_n}}^p + \frac{c_{p,\ee}}{n^{2\ee}} \,;\, \tau_y > \nu_n \,,\, \nu_n \leq \pent{n^{1-\ee}} \right).
\]
Since $\frac{y+S_{\nu_n}}{n^{1/2-\ee}}>1$, using Lemma \ref{lemsurE1},
\begin{equation}
	\label{BoundJ2}
	J_1 \leq \e^{-c_{p,\ee} n^{\ee/6}} \left(  1+\abs{x} \right)^p + \frac{c_{p,\ee}}{n^{1/2-\ee+2\ee}} E_1 \leq \frac{c_{p,\ee} (1+ y + \abs{x} )(1+\abs{x})^{p-1}}{n^{1/2+\ee}}.
\end{equation}

\textit{Bound of $J_2$.}
The idea is as follows.
When $y' \leq \theta_n \sqrt{n}$, with $\theta_{n} = n^{-\ee/6}$, we are going to control the probability 
$\bb P_{x'} \left(\tau_{y'} > n-k \,,\, A_{n-k} \right)$ in $J_2$
by the claim 2 of Corollary \ref{exittimeforB}. 
When $y' > \theta_n \sqrt{n}$ we shall apply Lemma \ref{lemsurE3}.
Accordingly, we split $J_2$ into two terms as follows: 
\begin{align}
	J_2 = &\sum_{k=1}^{\pent{n^{1-\ee}}} \int_{\bb R \times \bb R_+^*} \bb P_{x'} \left(\tau_{y'} > n-k \,,\, A_{n-k} \right) \bb P_x \left( X_k \in \dd x' \,,\, y+S_k \in \dd y' \,,\, \right. \nonumber\\
	&\underbrace{\hspace{5cm} \left. y+S_k > n^{1/2-\ee/6} \,,\, \tau_y > k \,,\, \nu_n^{\ee/6}=k \right)}_{=:J_3} \nonumber\\
	\label{decompositiondeJ1}
	+ &\sum_{k=1}^{\pent{n^{1-\ee}}} \int_{\bb R \times \bb R_+^*} \bb P_{x'} \left(\tau_{y'} > n-k \,,\, A_{n-k} \right) \bb P_x \left( X_k \in \dd x' \,,\, y+S_k \in \dd y' \,,\, \right. \\
	&\underbrace{\hspace{5cm} \left. y+S_k \leq n^{1/2-\ee/6} \,,\, \tau_y > k \,,\, \nu_n^{\ee/6} =k \right)}_{=:J_4}. \nonumber
\end{align}

\textit{Bound of $J_3$.} Let $\tau_y^{bm}$ be the exit time of the Brownian motion defined by (\ref{deftaubm}) and 
$y_+'=y'+(n-k)^{1/2-2\ee}$. 
Since 
\begin{equation}
\bb P_{x'} \left(\tau_{y'} > n-k \,,\, A_{n-k} \right) \leq \bb P_{x'} \left( \tau_{y_+'}^{bm} > n-k \right), 
\label{BoundJ4-000}
\end{equation}
using the claim 1 of Corollary \ref{exittimeforB} with $y_+' > 0$, we get
	\[
	J_3 \leq \sum_{k=1}^{\pent{n^{1-\ee}}} \bb E_x \left( c\frac{ y+S_k + (n-k)^{1/2-2\ee} }{\sqrt{n-k}} \,;\, y+S_k > n^{1/2-\ee/6} \,,\, \tau_y > k \,,\, \nu_n^{\ee/6} =k \right).
\]
Since $\frac{c}{\sqrt{n-k}} \leq \frac{c_{\ee}}{\sqrt{n}}$ and $y+S_k + (n-k)^{1/2-2\ee} \leq 2\left( y+S_k \right)$ on the event $\{ y+S_k > n^{1/2-\ee/6} \}$, using Lemma \ref{lemsurE3}, we have
\begin{equation}
\label{BoundJ4}
	J_3 \leq \frac{c_{\ee}}{\sqrt{n}} E_3 = o\left( \frac{1}{\sqrt{n}} \right).
\end{equation}
	
\textit{Upper bound of $J_4$.}
Since $\frac{n}{c_{\ee}} \leq n-k \leq n$, we have $y_+' \leq c_{\ee} (n-k)^{1/2-\ee/6}$ when $y' \leq n^{1/2-\ee/6}$. 
Using (\ref{BoundJ4-000}), from the claim 2 of Corollary \ref{exittimeforB} with $\theta_{m} = c_{\ee} m^{-\ee/6}$, we deduce that
\begin{align}
	J_4 &\leq \sum_{k=1}^{\pent{n^{1-\ee}}} \bb E_x \left( \frac{2}{\sqrt{2\pi \left(n-k\right) }\sigma} \left(y+S_k + \left( n-k \right)^{1/2-2\ee}\right) \left( 1+c \theta_{n-k}^2 \right) \,;\, \right. \nonumber\\
	&\hspace{4cm} \left. \phantom{\frac{2}{\sqrt{2\pi \left(n-k\right) }\sigma}} y+S_k \leq n^{1/2-\ee/6} \,,\, \tau_y > k \,,\, \nu_n^{\ee/6} =k \right).
	\label{Boutdedemo1}
\end{align}
Taking into account that $\frac{1}{\sqrt{n-k}} \leq \frac{1}{\sqrt{n}} \left( 1 + \frac{c_{\ee}}{n^{\ee}} \right)$, $\theta_{n-k} \leq \frac{c_{\ee}}{n^{\ee/6}}$ and $1 < \frac{y+S_{\nu_n}}{n^{1/2-\ee}}$, we obtain
\begin{equation}
	J_4 \leq \frac{2}{\sqrt{2\pi n}\sigma} \left( 1+\frac{c_{\ee}}{n^{\ee/3}} \right) E_2 + \frac{c_{\ee}}{n^{1/2+\ee}} E_1.
	\label{NouveauJ4}
\end{equation}
Using Lemma \ref{lemsurE1} and Lemma \ref{lemsurE2}, we get the following upper bound,
\begin{equation}
\label{UpBoundJ3}
	J_4 \leq \frac{2 V(x,y)}{\sqrt{2\pi n}\sigma} \left( 1+o(1) \right).
\end{equation}

\textit{Lower bound of $J_4$.}
In the same way as for the upper bound of $J_4$, with $y_-' = y + S_{\nu_n^{\ee/6}} - \left(n-\nu_n^{\ee/6}\right)^{1/2-2\ee} >0 $ on the event $\{ \left(n-\nu_n^{\ee/6}\right)^{1/2-2\ee} < y+S_{\nu_n^{\ee/6}} \}$, we have
\begin{align}
	J_4 \geq \; &\frac{2}{\sqrt{2\pi n}\sigma} \left( 1-\frac{c_{\ee}}{n^{\ee/6}} \right) \bb E_x \left( y_-' \,;\, (n-k)^{1/2-2\ee} < y+S_{\nu_n^{\ee/6}} \leq n^{1/2-\ee/6} \,,\, \right. \nonumber \\
	\label{Boutdedemo2}
	& \hspace{8cm} \left. \tau_y > \nu_n^{\ee/6} \,,\, \nu_n^{\ee/6} \leq \pent{n^{1-\ee}} \right) \\
	&-\sum_{k=1}^{\pent{n^{1-\ee}}} \int_{\bb R} \bb P_{x'} \left( \overline{A}_{n-k} \right) \bb P_x \left( X_k \in \dd x' \,,\, \tau_y > k \,,\, \nu_n^{\ee/6} =k \right). \nonumber
\end{align}
Using the fact that $-y_-'\geq 0$ on $\{ \left(n-\nu_n^{\ee/6}\right)^{1/2-2\ee} \geq y+S_{\nu_n^{\ee/6}} \}$, we obtain in a same way as for the upper bound of $J_1$,
\begin{align*}
	J_4 \geq\; &\frac{2}{\sqrt{2\pi n}\sigma} \left( 1-\frac{c_{\ee}}{n^{\ee/6}} \right) E_2 -\frac{2}{\sqrt{2\pi n}\sigma} \bb E_x \left( n^{1/2-2\ee} \frac{y+S_{\nu_n}}{n^{1/2-\ee}} \,;\, \tau_y > \nu_n \,,\, \nu_n \leq \pent{n^{1-\ee}} \right)\\
	&-\frac{2}{\sqrt{2\pi n}\sigma} E_3 -\frac{c_{p,\ee} (1+ y + \abs{x} )(1+\abs{x})^{p-1}}{n^{1/2+\ee}}\\
	\geq\; &\frac{2}{\sqrt{2\pi n}\sigma} \left( 1-\frac{c_{\ee}}{n^{\ee/6}} \right) E_2 -\frac{c}{n^{1/2+\ee}} E_1 -\frac{c}{\sqrt{n}} E_3 -\frac{c_{p,\ee} (1+ y + \abs{x} )(1+\abs{x})^{p-1}}{n^{1/2+\ee}}.
\end{align*}
Consequently, using the results of Lemma \ref{lemsurE2}, Lemma \ref{lemsurE1} and Lemma \ref{lemsurE3} we conclude that
\begin{equation}
\label{LoBoundJ3}
	J_4 \geq \frac{2V(x,y)}{\sqrt{2\pi n}\sigma} \left( 1-o(1) \right).
\end{equation}
Coupling the obtained lower bound with the upper bound in (\ref{UpBoundJ3}) we obtain
$J_4 \sim  \frac{2V(x,y)}{\sqrt{2\pi n}\sigma}.$
With the decomposition of $J_2$ in (\ref{decompositiondeJ1}) and the bound of $J_3$ in (\ref{BoundJ4}) we get $J_2 \sim  \frac{2V(x,y)}{\sqrt{2\pi n}\sigma}.$ 
Finally, the claim 2 of Theorem \ref{AsExTi} follows from (\ref{DecompJ0-001}), (\ref{BoundJ0-002}), (\ref{decompositiondeT})  and (\ref{BoundJ2}).

\subsection{Proof of the claim 1 of Theorem \ref{AsExTi}}

All the bounds necessary are obtained in the proofs of the previous section \ref{claim2thasexti}. We highlight how to gather it.
By (\ref{DecompJ0-001}), (\ref{BoundJ0-002}), (\ref{decompositiondeT}) and (\ref{decompositiondeJ1}), we have,
\[ \bb P_x \left( \tau_y > n \right) \leq c_{p,\ee} \frac{ \left(1+\abs{x}^p \right) }{ \sqrt{n} } +J_1+J_3+J_4. \]
Then, by (\ref{BoundJ2}), (\ref{BoundJ4}), and (\ref{NouveauJ4}),
\[ \bb P_x \left( \tau_y > n \right) \leq c_{p,\ee} \frac{ \left(1+y+\abs{x} \right)\left( 1+\abs{x} \right)^{p-1} }{ \sqrt{n} }  + \frac{c_\ee}{\sqrt{n}} E_3 + \frac{c_\ee}{\sqrt{n}}\left( E_2 +E_1 \right). \]
Now, by Lemma \ref{lemsurE1}, (\ref{decompositiondeE2}) and (\ref{belmajdeE3}),
\[ \bb P_x \left( \tau_y > n \right) \leq c_{p,\ee} \frac{ \left(1+y+\abs{x} \right)^p }{\sqrt{n}} + \frac{c_\ee}{\sqrt{n}} \left( E_{21} + E_{22} \right). \]
Finally, using (\ref{BoundE22}), (\ref{majdeE22}) and Lemmas \ref{intdeMtau} and \ref{intdeMtaucasneg} we have,
\begin{align*}
	\bb P_x \left( \tau_y > n \right) &\leq \frac{c_\ee}{\sqrt{n}} \left( z - \bb E_x \left( z+M_{\tau_y} \,;\, \tau_y \leq \nu_n^{\ee/6} \wedge \pent{n^{1-\ee}} \right) \right) + c_{p,\ee} \frac{ \left(1+y+\abs{x} \right)^p }{\sqrt{n}} \\
	&\leq \frac{c_\ee}{\sqrt{n}} \bb E_x \left( \abs{M_{\tau_y}} \right) + c_{p,\ee} \frac{ \left(1+y+\abs{x} \right)^p }{\sqrt{n}} \\
	&\leq c_{p} \frac{ \left(1+y+\abs{x} \right)^p }{\sqrt{n}}.
\end{align*}

\subsection{Proof of Corollary \ref{MoofExTi}}

By the Fubini theorem, for any $1/2 \geq p >0$,
	\[
	\bb E_x \left( \tau_y^{p} \right) = \int_{0}^{+\infty} \bb P_x \left( \tau_y > s \right) p s^{p-1} \dd s = \sum_{k=0}^{+\infty} \bb P_x \left( \tau_y > k \right) \left( (k+1)^p - k^p \right).
\]
Using Theorem \ref{AsExTi}, the sum $\sum_{k=1}^{+\infty} \frac{1}{k^{1+1/2-p}}$ is finite if and only if $1/2-p>0$.

\section{Asymptotic for conditioned Markov walk}
\label{As for cond Markov walk}

In this section we prove Theorem \ref{AsCoRaWa}. We will deduce the asymptotic of the Markov walk $\left( y+S_n \right)_{n\geq 0}$ conditioned to stay positive from the corresponding result for  the Brownian motion given by Proposition \ref{intBro}. 
As in Section \ref{As for Exit Time}, we will use the functional approximation of Proposition \ref{majdeA_k}. 
We will refer frequently to Section \ref{As for Exit Time} in order to shorten the exposition.

\textbf{Proof of Theorem \ref{AsCoRaWa}.} Introducing $\nu_n^{\ee/6} = \nu_n + \pent{n^{\ee/6}}$ and taking into account Condition \ref{PosdeX1-1}, we have
\begin{align}
	\bb P_x \left( \sachant{ y+S_n \leq t\sqrt{n}}{\tau_y > n} \right) =\; &\underbrace{\frac{\bb P_x \left( y+S_n \leq t\sqrt{n} \,,\, \tau_y > n \,,\, \nu_n^{\ee/6} > \pent{n^{1-\ee}} \right)}{\bb P_x \left( \tau_y > n \right)}}_{=:L_1} \nonumber\\
	&\qquad + \underbrace{\frac{\bb P_x \left( y+S_n \leq t\sqrt{n} \,,\, \tau_y > n \,,\, \nu_n^{\ee/6} \leq \pent{n^{1-\ee}} \right)}{\bb P_x \left( \tau_y > n \right)}}_{=:L_2}. \label{Decompositioninitiale}
\end{align}

\textit{Bound of $L_1$.} Using Lemma \ref{concentnu} and Theorem \ref{AsExTi},
\begin{equation}
  \label{Boundinitiale}
	L_1 \leq \frac{\bb P_x \left( \nu_n > \frac{n^{1-\ee}}{c_\ee} \right)}{\bb P_x \left( \tau_y > n \right)} \leq \frac{c_{p,\ee} \left(1+\abs{x}\right)^p}{n^{\frac{p}{2}-c_p \ee}\bb P_x \left( \tau_y > n \right)} \underset{n\to +\infty}{\longrightarrow} 0.
\end{equation}

\textit{Bound of $L_2$.} As in Section \ref{As for Exit Time}, setting $A_k = \left\{ \underset{0\leq t \leq 1}{\max} \abs{ S_{\pent{tk}} - \sigma B_{tk} } \leq k^{1/2-2\ee} \right\}$, by the Markov property,
\begin{align}
	\bb P_x &\left( \tau_y > n \right)L_2 \nonumber\\
	=\; &\sum_{k=1}^{\pent{n^{1-\ee}}} \int_{\bb R \times \bb R_+^*} \bb P_{x'} \left( y'+S_{n-k} \leq t\sqrt{n} \,,\, \tau_{y'} > n-k \,,\, \overline{A}_{n-k} \right) \bb P_x \left( X_k \in \dd x' \,,\, \right. \nonumber\\
	&\underbrace{\hspace{7cm} \left. y+S_k \in \dd y' \,,\, \tau_y > k \,,\, \nu_n^{\ee/6} =k \right)}_{=:\bb P_x \left( \tau_y > n \right) L_3} \nonumber\\
	\label{DecompositiondeL1}
	&+\sum_{k=1}^{\pent{n^{1-\ee}}} \int_{\bb R \times \bb R_+^*} \bb P_{x'} \left( y'+S_{n-k} \leq t\sqrt{n} \,,\, \tau_{y'} > n-k \,,\, A_{n-k} \right) \bb P_x \left( X_k \in \dd x' \,,\, \right. \\
	&\underbrace{\hspace{4cm} \left. y+S_k \in \dd y' \,,\, y+S_k > n^{1/2-\ee/6} \,,\, \tau_y > k \,,\, \nu_n^{\ee/6} =k \right)}_{=:\bb P_x \left( \tau_y > n \right) L_4} \nonumber\\
	&+\sum_{k=1}^{\pent{n^{1-\ee}}} \int_{\bb R \times \bb R_+^*} \bb P_{x'} \left( y'+S_{n-k} \leq t\sqrt{n} \,,\, \tau_{y'} > n-k \,,\, A_{n-k} \right) \bb P_x \left( X_k \in \dd x' \,,\, \right. \nonumber\\
	&\underbrace{\hspace{4cm} \left. y+S_k \in \dd y' \,,\, y+S_k \leq n^{1/2-\ee/6} \,,\, \tau_y > k \,,\, \nu_n^{\ee/6} = k \right)}_{=:\bb P_x \left( \tau_y > n \right) L_5}. \nonumber
\end{align}

\textit{Bound of $L_3$.} Using the bound of $J_1$ in (\ref{BoundJ2}) and Theorem \ref{AsExTi},
\begin{equation}
\label{BoundL5}
	L_3 \leq \frac{J_1}{\bb P_x \left( \tau_y > n \right)} \leq \frac{c_{p,\ee} (1+ y + \abs{x} )(1+\abs{x})^{p-1}}{n^{1/2+\ee} \bb P_x \left( \tau_y > n \right)} \underset{n\to +\infty}{\longrightarrow} 0.
\end{equation}

\textit{Bound of $L_4$.} 
Using the bound of $J_3$ in (\ref{BoundJ4}) and Theorem \ref{AsExTi}, we have
\begin{equation}
  \label{BoundL4}
	L_4 \leq \frac{J_3}{\bb P_x \left( \tau_y > n \right)}= o(1).
\end{equation}

\textit{Upper bound of $L_5$.}
Define $t_+ =t+\frac{2}{(n-k)^{2\ee}}$
and $y_+' = y' + (n-k)^{1/2-2\ee}$. By Proposition \ref{intBro},
\begin{align*}
	\bb P_{x'} &\left( y'+S_{n-k} \leq t\sqrt{n} \,,\, \tau_{y'} > n-k \,,\, A_{n-k} \right) \\
	&\hspace{2cm} \leq \bb P \left( y_+' + \sigma B_{n-k} \leq t_+ \sqrt{n} \,,\, \tau_{y_+'}^{bm} > n-k \right) \\
	&\hspace{2cm} = \frac{1}{\sqrt{2\pi (n-k)} \sigma} \int_0^{t_+ \sqrt{n}} \e^{-\frac{(s-y_+')^2}{2(n-k) \sigma^2}} - \e^{-\frac{(s+y_+')^2}{2(n-k) \sigma^2}} \dd s.
\end{align*}
Note that $y_+'/\sqrt{n} \leq \frac{2}{n^{\ee/6}}$ when $y' \leq n^{1/2-\ee/6}$ and that for any $k \leq \pent{n^{1-\ee}}$ we have $1-\frac{c_\ee}{n^\ee} \leq n-k \leq n$. Using these remarks with the fact that $\abs{\sh(x) - x} \leq \frac{x^3}{6} \sh(x)$, we obtain after some calculations that
	\[
	\bb P_{x'} \left( y'+S_{n-k} \leq t\sqrt{n} \,,\, \tau_{y'} > n-k \,,\, A_{n-k} \right) \leq \frac{2 y_+'}{\sqrt{2\pi n} \sigma} \left(1+\frac{c_{t,\ee}}{n^{\ee/3}}\right)  \left( 1 -\e^{-\frac{t^2}{2 \sigma^2}} \right).
\]
Consequently, using the same arguments as in the proof of Theorem \ref{AsExTi} in Section \ref{As for Exit Time} (see the developments from (\ref{Boutdedemo1}) to (\ref{UpBoundJ3})), 
we obtain
\begin{equation}
\label{UpBoundL3}
	L_5 \leq \left(1+\frac{c_{t,\ee}}{n^{\ee/3}}\right) \mathbf \Phi_{\sigma}^+(t) \frac{2V(x,y)}{\sqrt{2\pi n} \sigma \bb P_x \left( \tau_y > n \right)} \left(1+o(1)\right) = \mathbf \Phi_{\sigma}^+(t) \left(1+o(1)\right),
\end{equation}
with $\mathbf \Phi_{\sigma}^+(t) = 1-\e^{-\frac{t^2}{2\sigma^2}}$.

\textit{Lower bound of $L_5$.} In the same way as for the upper bound, with $y_-' = y' - (n-k)^{1/2-2\ee}$ and $t_- = t-\frac{2}{(n-k)^{2\ee}}$, we have
\begin{align*}
	\bb P_x \left( \tau_y > n \right)& L_5 \\
	\geq\;& \sum_{k=1}^{\pent{n^{1-\ee}}} \int_{\bb R_+^*} \bb P \left( y_-' + \sigma B_{n-k} \leq t_- \sqrt{n} \,,\, \tau_{y_-'}^{bm} > n-k \right) \bb P_x \left( y+S_{n-k} \in \dd y'  \,,\, \right. \\
	&\hspace{2cm} \left.  (n-k)^{1/2-2\ee} < y+S_{k} \leq n^{1/2-\ee/6} \,,\, \tau_y > k \,,\, \nu_n^{\ee/6} =k \right) \\
	&-\sum_{k=1}^{\pent{n^{1-\ee}}} \int_{\bb R} \bb P_{x'} \left( \overline{A}_{n-k} \right) \bb P_x \left( X_k \in \dd x' \,,\, \tau_y > k \,,\, \nu_n^{\ee/6}=k \right).
\end{align*}
Using Lemma \ref{intBro} with $y_-'$, which is positive when $(n-k)^{1/2-2\ee} < y' \leq n^{1/2-\ee/6}$, we obtain after calculation that
	\[
	\bb P \left( y_-' + \sigma B_{n-k} \leq t_- \sqrt{n} \,,\, \tau_{y_-'}^{bm} > n-k \right) \geq \frac{2y_-'}{\sqrt{2\pi n} \sigma} \left(1-\frac{c_{t,\ee}}{n^{\ee/3}}\right) \mathbf \Phi_{\sigma}^+(t). 
\]
Copying the proof of the bound of $J_1$ in (\ref{BoundJ2}) and using the same arguments as in the proof 
of Theorem \ref{AsExTi} in Section \ref{As for Exit Time} (see the developments from (\ref{Boutdedemo2}) to (\ref{LoBoundJ3})), we get
	\[
	L_5 \geq \mathbf \Phi_{\sigma}^+(t) \frac{2V(x,y)}{\sqrt{2\pi n}\sigma \bb P_x \left( \tau_y > n \right)} \left( 1-o(1) \right) = \mathbf \Phi_{\sigma}^+(t) \left( 1-o(1) \right).
\]
Coupling this with (\ref{UpBoundL3}) we obtain that
	\[
	L_5 = \mathbf \Phi_{\sigma}^+(t) \left( 1+o(1) \right).
\]
Inserting this and (\ref{BoundL5}) and (\ref{BoundL4}) into (\ref{DecompositiondeL1}), we deduce that $L_2 \underset{n\to+\infty}{\sim} \mathbf \Phi_{\sigma}^+(t)$. By (\ref{Decompositioninitiale}) and (\ref{Boundinitiale}), we finally have
	\[
	\bb P_x \left( \sachant{ y+S_n \leq t\sqrt{n}}{\tau_y > n} \right) \underset{n \to +\infty}{\longrightarrow} \mathbf \Phi_{\sigma}^+(t).
\]
Changing $t$ into $t\sigma$, this concludes the proof.

\section{The case of non-positive initial point}
\label{secproofofinitialpointneg}
In this section, we prove Theorem \ref{initialpointneg}. All over this section we assume either Conditions \ref{Mom001}, \ref{PosdeX1-1} and  $\bb E(a) \geq 0$, or Conditions \ref{Mom001} and \ref{PosdeX1-2}.

\begin{lemma}
\label{Corpartie1}
For any $(x,y) \in \mathscr{D}^-$, the random variable $M_{\tau_y}$ is integrable and the function 
$V(x,y) = -\bb E_x \left( M_{\tau_y} \right)$, is well defined on $\mathscr{D}^-$.
\end{lemma}

\begin{proof}
If $\bb E(a) \geq 0$, by the Markov inequality, with $z=y+\rho x$,
\begin{align*}
	\bb E_x \left( z+M_n \,;\, \tau_y > n \right) &= \int_{\bb R \times \bb R_+^*} \bb E_{x'} \left( y'+\rho x'+M_{n-1} \,;\, \tau_{y'} > n-1 \right) \\
	&\hspace{4cm} \times \bb P_x \left( X_1 \in \dd x' \,,\, y+S_1 \in \dd y' \,,\, \tau_y > 1 \right).
\end{align*}
Since $y+S_1>0$ on $\{ \tau_y > 1 \}$, by Lemma \ref{intofthemartcond}, 
\begin{align}
	\bb E_x \left( z+M_n \,;\, \tau_y > n \right) &\leq c_p \bb E_x \left( \left( 1+y+S_1+\abs{X_1} \right) \left( 1+ \abs{X_1}\right)^{p-1} \,;\, \tau_y > 1 \right) \nonumber\\
	&\leq c_p \bb E_x \left( \left( 1+\abs{X_1} \right)^p \right) \nonumber\\
	&\leq c_p \left( 1+\abs{x} \right)^p.
	\label{majVncasyneg}
\end{align}
Moreover
\begin{align*}
	\bb E_x \left( \abs{M_{\tau_y}} \,;\, \tau_y \leq n \right) \leq \abs{z} &+ \sum_{k=2}^n \int_{\bb R \times \bb R_+^*} \bb E_{x'} \left( \abs{y'+\rho x' + M_{k-1}} \,;\, \tau_y = k-1 \right) \\
	&\hspace{2cm}\times \bb P_x \left( X_1 \in \dd x' \,,\, y+S_1 \in \dd y' \,,\, \tau_y > 1 \right) \\
	&+ \bb E_x \left( \abs{M_1} \,;\, \tau_y = 1 \right).
\end{align*}
Since $y+S_1 > 0$ on $\{\tau_y > 1\}$, by Lemma \ref{Mtauyprop},
\begin{align*}
	\bb E_x \left( \abs{M_{\tau_y}} \,;\, \tau_y \leq n \right) &\leq  c\left( 1+\abs{y}+\abs{x} \right) - \bb E_{x} \left( z + M_{\tau_y} \,;\, \tau_y \leq n \right) \\
	&\leq c\left( 1+\abs{y}+\abs{x} \right) + \bb E_{x} \left( z + M_n \,;\, \tau_y > n \right).
\end{align*}
Using (\ref{majVncasyneg}), we deduce that $\bb E_x \left( \abs{M_{\tau_y}} \,;\, \tau_y \leq n \right) \leq c_p \left( 1+\abs{y}+\abs{x}^p \right)$. Consequently, by the Lebesgue monotone convergence theorem, the assertion is proved when $\bb E(a) \geq 0$.
When $\bb E(a) <0$, the assertion follows from Lemma \ref{intdeMtaucasneg}.
\end{proof}

\begin{lemma}
\label{Corpartie2}
The function $V$ is $\textbf{Q}_+$-harmonic and positive on $\mathscr{D} = \mathscr{D}^- \cup \bb R \times \bb R_+^*$.
\end{lemma}

\begin{proof}
Note that by Corollary \ref{Exitfinit}, we have $\bb P_x (\tau_y < +\infty)=1$, for any  $x \in \bb R$ and $y \in \bb R$.
Therefore, by the Lebesgue dominated convergence theorem,
\[ V(x,y) = -\bb E_x \left( M_{\tau_y} \right) = z - \underset{n\to \infty}{\lim} \bb E_x \left( z+M_{\tau_y} \,;\, \tau_y \leq n \right) = \underset{n\to \infty}{\lim} \bb E_x \left( z+M_n \,;\, \tau_y > n \right), \]
for any $(x,y) \in \mathscr{D}^-$. The fact that $V$ is $\textbf{Q}_+$-harmonic on $\mathscr{D}$ can be proved 
in the same way as in the proof of Proposition \ref{HaetposdeV}.
Therefore, for any $(x,y) \in \mathscr{D}^-$,
\begin{equation}
	\label{harmyneg}
	V(x,y) = \bb E_x \left( V(X_1, y+S_1) \,;\, \tau_y > 1 \right).
\end{equation}
By the claim 2 of Proposition \ref{HaetposdeV} and the claim 3 of Lemma \ref{Invfunctpropcasnegpart2}, on $\{ \tau_y > 1\}$, the random variable $V(X_1, y+S_1)$ is positive almost surely. Since by the definition of $\mathscr{D}^-$, we have 
$\bb P_x \left( \tau_y > 1 \right) >0$, we conclude that $V(x,y) >0$ for any $(x,y) \in \mathscr{D}^-$.
\end{proof}

\begin{lemma}\ 
\label{Corpartie3}
\begin{enumerate}
\item For any $(x,y) \in \mathscr{D}^-$,
\[ \sqrt{n}\bb P_x \left( \tau_y > n \right) \leq c_p \left( 1+ \abs{x} \right)^{p}. \]
\item For any $(x,y) \in \mathscr{D}^-$,
\[ \bb P_x \left( \tau_y > n \right)  \underset{n\to +\infty}{\sim} \frac{2V(x,y)}{\sqrt{2\pi n} \sigma}. \]
\end{enumerate}
\end{lemma}

\begin{proof}
By the Markov property,
\[ \sqrt{n} \bb P_x \left( \tau_y >n \right) = \int_{\bb R\times \bb R_+^*} \sqrt{n} \bb P_{x'} \left( \tau_{y'} > n-1 \right) \bb P_x \left( X_1 \in \dd x' \,,\, y+S_1 \in \dd y' \,,\, \tau_y > 1 \right). \]
By  Theorem \ref{AsExTi}, for any $y'> 0$, we have $\sqrt{n} \bb P_{x'} \left( \tau_{y'} > n-1 \right) \leq c_p \left( 1+ y'+\abs{x'} \right)^{p}$ and moreover, for any $y\leq 0$,
\[ \bb E_x \left( c_p \left( 1+ y+S_1+\abs{X_1} \right)^{p} \,;\, \tau_y > 1 \right) \leq c_p \left( 1+\abs{x} \right)^p. \]
Then, we obtain the claim 1 and by the Lebesgue dominated convergence theorem and the claim 2 of Theorem \ref{AsExTi},
\[ \underset{n \to \infty}{\lim} \sqrt{n} \bb P_x \left( \tau_y >n \right) = \bb E_x \left( \frac{2 V(X_1,y+S_1) }{\sqrt{2\pi}\sigma} \,;\, \tau_y > 1 \right). \]
Using (\ref{harmyneg}) we conclude the proof.
\end{proof}

\begin{lemma}
\label{Corpartie4}
For any $(x,y) \in \mathscr{D}^-$ and $t>0$,
\[
	\bb P_x \left( \sachant{ \frac{y+S_n}{\sigma \sqrt{n}} \leq t}{\tau_y > n} \right)  \underset{n\to +\infty}{\longrightarrow} 1-\e^{-\frac{t^2}{2}}.
\]
\end{lemma}

\begin{proof}
Similarly as in the proof of Lemma \ref{Corpartie3}, we write,
\begin{align*}
	\bb P_x &\left( \sachant{ \frac{y+S_n}{\sigma \sqrt{n}} \leq t}{\tau_y > n} \right) \\
	&= \frac{1}{\bb P_x \left( \tau_y >n \right) }\int_{\bb R\times \bb R_+^*} \bb P_{x'} \left( \frac{y'+S_{n-1}}{\sigma \sqrt{n-1}} \leq t \,;\, \tau_{y'} > n-1 \right) \\
	&\hspace{4cm} \times \bb P_x \left( X_1 \in \dd x' \,,\, y+S_1 \in \dd y' \,,\, \tau_y > 1 \right) \\
	&= \frac{1}{\sqrt{n}\bb P_x \left( \tau_y >n \right) } \int_{\bb R\times \bb R_+^*} \bb P_{x'} \left( \sachant{\frac{y'+S_{n-1}}{\sigma \sqrt{n-1}} \leq t}{\tau_{y'} > n-1} \right) \sqrt{n} \bb P_{x'} \left( \tau_{y'} > n-1 \right)\\
	&\hspace{4cm}  \times\bb P_x \left( X_1 \in \dd x' \,,\, y+S_1 \in \dd y' \,,\, \tau_y > 1 \right).
\end{align*}
Since, by Lemma \ref{Corpartie3}, $\sqrt{n} \bb P_{x'} \left( \tau_{y'} > n-1 \right) \leq c_p \left( 1+\abs{x'} \right)^p$, applying the Lebesgue dominated convergence theorem, Theorem \ref{AsExTi}, Theorem \ref{AsCoRaWa} and Lemma \ref{Corpartie3}, we have
\begin{align*}
\underset{n \to \infty}{\lim} &\bb P_x \left( \sachant{ \frac{y+S_n}{\sigma \sqrt{n}} \leq t}{\tau_y > n} \right) \\
&= \frac{\sqrt{2\pi} \sigma}{2V(x,y)} \int_{\bb R\times \bb R_+^*} \left( 1-\e^{-\frac{t^2}{2}} \right) \frac{2V(x',y')}{\sqrt{2\pi} \sigma} \bb P_x \left( X_1 \in \dd x' \,,\, y+S_1 \in \dd y' \,,\, \tau_y > 1 \right).
\end{align*}
Using (\ref{harmyneg}) concludes the proof.
\end{proof}

\section{Appendix}
\label{Appendix}

\subsection{Proof of the fact Condition \ref{CSPosdeX1-2} implies Condition \ref{PosdeX1-2}}
\label{ComplementCond}

We suppose that Condition \ref{CSPosdeX1-2} holds true. Then, there exists $\delta > 0$ such that
\begin{align}
	\label{abbien001}
	\bb P \left( (a,b) \in [-1+\delta,0] \times [\delta,C] \right) > 0
\end{align}
and
\begin{align}
	\label{abbien002}
	\bb P \left( (a,b) \in [0,1-\delta] \times [\delta,C] \right) > 0.
\end{align}
For any $x \in \bb R$, set $C_x  = \max \left( \abs{x}, \frac{C}{\delta} \right)$ and
	\[
	\mathscr{A}_n = \left\{ \delta \leq X_1 \leq C_x \,,\, \delta \leq X_2 \leq C_{X_1} \,,\, \dots \,,\, \delta \leq X_n \leq C_{X_{n-1}} \right\}.
\]
Using (\ref{abbien001}) for $x < 0$ and (\ref{abbien002}) for $x \geq 0$, 
we obtain that $\bb P_x \left( \mathscr{A}_1 \right) >0$.
By the Markov property, we deduce that $\bb P_x \left( \mathscr{A}_n \right) >0$. 
Moreover, it is easy to see that, on $\mathscr{A}_n$, we have  $y+S_k \geq y + k\delta > 0$, for all $k \leq n$, and $\abs{X_n} \leq C_x$. Taking $n=n_0$ large enough, we conclude that Condition \ref{PosdeX1-2} holds under Condition \ref{CSPosdeX1-2}.

\subsection{Convergence of recursively bounded monotonic sequences}

The following lemmas give sufficient conditions for a monotonic sequence to be bounded.

\begin{lemma}
\label{lemanalyse}
Let $(u_n)_{n\geq 1}$ be a non-decreasing sequence of reals such that there exist $\ee \in (0,1)$ and $\alpha, \beta, \gamma, \delta \geq 0$ such that for any $n\geq 2$,
\begin{equation}
\label{recun}
u_n \leq \left( 1+ \frac{\alpha}{n^\ee} \right) u_{\pent{n^{1-\ee}}} + \frac{\beta}{n^\ee} + \gamma \e^{-\delta n^\ee}.
\end{equation}
Then, for any $n\geq 2$ and any integer $n_f \in \{2, \dots, n \}$,
\begin{align*} 
u_n &\leq \exp \left({ \frac{\alpha}{n_f^\ee}  \frac{2^\ee 2^{\ee^2}}{2^{\ee^2}-1} }\right) \left( u_{n_f} + \frac{\beta}{n_f^\ee} \frac{2^\ee 2^{\ee^2}}{2^{\ee^2}-1} + \gamma \frac{ \exp \left({ -\delta \frac{n_f^\ee}{2^\ee} }\right) }{ 1-\e^{ -\delta \left( 2^{\ee^2}-1 \right) } } \right) \\
&\leq \left( 1+ \frac{c_{\alpha,\ee}}{n_f^\ee} \right) u_{n_f} + \beta \frac{c_{\alpha,\ee}}{n_f^\ee} + \gamma \e^{-c_{\alpha,\delta,\ee} n_f^\ee}.
\end{align*}
In particular, choosing $n_f$ constant, it follows that $(u_n)_{n\geq 1}$ is bounded. 
\end{lemma}

\begin{proof}
Fix $n\geq 2$ and $n_f \in \{2, \dots, n\}$ and consider for all $j \geq 0$,
\[ n_j = \pent{n^{(1-\ee)^j}}. \]
The sequence $(n_j)_{j\geq 0}$ starts at $n_0=n$, is non-increasing and converge to $1$. So there exists $m=m(n_f) \in \bb N$ such that $n_m \geq n_f \geq n_{m+1}$. Since $n^{(1-\ee)^j}/2 \geq n_f/2 \geq 1$, for all $j \in \{0, \dots, m\}$, we have
\begin{equation}
 \label{encadrenj} n^{(1-\ee)^j} \geq n_j \geq n^{(1-\ee)^j} - 1 \geq \frac{n^{(1-\ee)^j}}{2}.
\end{equation}
Using $(\ref{recun})$ and the fact that $(u_n)_{n\geq 2}$ is non-decreasing, we write for all $j =0,\dots, m$,
\[ u_{n_j} \leq \left( 1+ \frac{\alpha}{n_j^\ee} \right) u_{n_{j+1}} + \frac{\beta}{n_j^\ee} + \gamma \e^{-\delta n_j^\ee}
\leq \left( 1+ \frac{\alpha}{n_j^\ee} \right) \left( u_{n_{j+1}} + \frac{\beta}{n_j^\ee} + \gamma \e^{-\delta n_j^\ee} \right). \]
Iterating, we obtain that
\[ u_n \leq A_m \left( u_{n_{m+1}} + \beta B_m + \gamma C_m \right), \]
where $A_m  = \prod_{j=0}^m \left( 1+ \frac{\alpha}{n_j^\ee} \right)$, $B_m  = \sum_{j=0}^m \frac{1}{n_j^\ee}$ and $C_m = \sum_{j=0}^m \e^{-\delta n_j^\ee}$. Since $n_{m+1} \leq n_f$ and since $(u_n)_{n\geq 2}$ is non-decreasing,
\begin{equation}
	\label{recunj}
	u_n \leq A_m \left( u_{n_{n_f}} + \beta B_m + \gamma C_m \right).
\end{equation}
Now, we bound $A_m$ as follows,
\begin{equation}
	\label{majAm}
	A_m \leq \prod_{j=0}^m \e^{\frac{\alpha}{n_j^\ee}} = \e^{\alpha B_m}.
\end{equation}
Denoting $\eta_j = n^{-(1-\ee)^j\ee}$, using (\ref{encadrenj}), we have $B_m \leq 2^\ee \sum_{j=0}^m \eta_j$. Moreover, for all $j \leq m$, we note that $\frac{\eta_j}{\eta_{j+1}} = \frac{1}{n^{\ee^2 (1-\ee)^j}} \leq \frac{1}{n_f^{\ee^2}} \leq \frac{1}{2^{\ee^2}} <1$ and so
\begin{equation}
	\label{majetaj}
	\eta_j \leq \frac{\eta_m}{2^{\ee^2(m-j)}} \leq \frac{1}{n_m^\ee 2^{\ee^2(m-j)}} \leq \frac{1}{n_f^\ee 2^{\ee^2(m-j)}}.
\end{equation}
Therefore, $B_m$ is bounded as follows:
\begin{equation}
	\label{majBm}
	B_m \leq \frac{2^\ee}{n_f^\ee} \sum_{k=0}^m \left( \frac{1}{2^{\ee^2}} \right)^k \leq \frac{1}{n_f^\ee} \frac{2^\ee 2^{\ee^2}}{2^{\ee^2}-1}.
\end{equation}
Using (\ref{encadrenj}) and (\ref{majetaj}), we have
\[ C_m \leq \sum_{j=0}^m \e^{-\frac{\delta}{2^\ee \eta_j}} \leq \sum_{j=0}^m \exp \left({-\frac{\delta n_f^\ee 2^{\ee^2(m-j)}}{2^\ee}}\right). \]
Since for any $u \geq 0$ and $k \in \bb N$, we have $(1+u)^k \geq 1+ku$, it follows that
\begin{equation}
	\label{majCm}
	C_m \leq \e^{ -\frac{\delta n_f^\ee}{2^\ee} } \sum_{k=0}^m \exp\left({-{\delta k \left(2^{\ee^2}-1\right)}}\right)
	\leq \frac{\e^{ -\frac{\delta n_f^\ee}{2^\ee} }}{1-\e^{-\delta  \left(2^{\ee^2}-1\right)}}.
\end{equation}reals
Putting together (\ref{majAm}), (\ref{majBm}) and (\ref{majCm}) into (\ref{recunj}), proves the lemma.
\end{proof} 

\begin{lemma}
\label{lemanalyse2}
Let $(u_n)_{n\geq 1}$ be a non-increasing sequence of reals such that there exist $\ee \in (0,1)$ and $\beta \geq 0$ such that for any $n\geq 2$,
\[
u_n \geq u_{\pent{n^{1-\ee}}} - \frac{\beta}{n^\ee}.
\]
Then, for any $n\geq 2$ and any integer $n_f \in \{2, \dots, n \}$,
\[
u_n \geq u_{n_f} - \frac{\beta}{n_f^\ee} \frac{2^\ee 2^{\ee^2}}{2^{\ee^2}-1} = u_{n_f} - c_\ee \frac{\beta}{n_f^\ee}.
\]
In particular, choosing $n_f$ constant, it follows that  $(u_n)_{n\geq 1}$ is bounded.
\end{lemma}
\begin{proof} For the proof it is enough to use Lemma \ref{lemanalyse} with $u_n$ replaced by $-u_n$.
\end{proof}

\subsection{Results on the Brownian case and strong approximation}\label{Strong Approx}

Consider the standard Brownian motion $\left( B_t \right)_{t\geq 0}$ living on the probability space $\left( \Omega, \mathcal F, \pmb{\mathbf{\bb P}} \right)$. 
Define the exit time
\begin{equation}
	\label{deftaubm}
	\tau_y^{bm} = \inf \{ t\geq 0, \, y+\sigma B_t \leq 0 \},
\end{equation}
where $\sigma>0.$ The following assertions are due to Levy \cite{levy_theorie_1954}.
\begin{proposition}
\label{intBro}
For any $y>0$, $0\leq a \leq b$ and $n \geq 1$,
	\[
	\pmb{\bb P} \left( \tau_y^{bm} > n \right) = \frac{2}{\sqrt{2\pi n} \sigma} \int_{0}^{y} \e^{-\frac{s^2}{2n\sigma^2}} \dd s.
\]
and
	\[
	\pmb{\bb P} \left( \tau_y^{bm} > n \,,\, y+\sigma B_n \in [a,b] \right) 
	= \frac{1}{\sqrt{2\pi n} \sigma} \int_a^b \left( \e^{-\frac{(s-y)^2}{2n\sigma^2}} - \e^{-\frac{(s+y)^2}{2n\sigma^2}} \right) \dd s.
\]
\end{proposition}
From this one can deduce easily:

\begin{corollary}
\label{exittimeforB} \ 
\begin{enumerate}
	\item For any $y>0$,
		\[
		\pmb{\bb P} \left( \tau_y^{bm}>n \right) \leq c\frac{y}{\sqrt{n}}.
	\]
	\item For any sequence of real numbers $(\theta_n)_{n\geq 0}$ such that $\theta_n \underset{n\to +\infty}{\longrightarrow} 0$,
		\[
		\underset{y\in [0; \theta_n \sqrt{n}]}{\sup} \left( \frac{\pmb{\bb P} \left( \tau_y^{bm}>n \right)}{\frac{2y}{\sqrt{2\pi n}\sigma}} - 1 \right) = O(\theta_n^2).
	\]
\end{enumerate}
\end{corollary}

To transfer the results from the Brownian motion to the Markov walk, we use a functional approximation given in Theorem 3.3  from 
Grama, Le Page and Peign\'{e} \cite{ion_grama_rate_2014}.
We have to construct an adapted Banach space  $\mathcal{B}$ and verify the hypotheses $\mathbf{M1-M5}$ in \cite{ion_grama_rate_2014} which are necessary to apply Theorem 3.3. 
Fix $p \in (2,\alpha)$ and let $\ee$, $\theta$, $c_0$ and $\delta$ be positive numbers such that 
$c_0+\ee < \theta < 2c_0 < \alpha - \ee$ 
and $2<2+2\delta  < (2+2\delta) \theta \leq p$. 
Define the Banach space $\mathcal{B}=\mathcal{L}_{\ee,c_0,\theta}$ as the set of continuous function $f$ from $\bb R$ to $\bb C$ such that $\norm{f} = \abs{f}_{\theta} + \left[ f \right]_{\ee,c_0} < +\infty$, where
\[
\abs{f}_{\theta} = \underset{x\in \bb R}{\sup} \frac{\abs{f(x)}}{1+\abs{x}^{\theta}}, \qquad \left[ f \right]_{\ee,c_0} =\underset{\substack{(x,y)\in \bb R^2\\x\neq y}}{\sup} \frac{\abs{f(x)-f(y)}}{\abs{x-y}^\ee \left(1+\abs{x}^{c_0}\right) \left(1+\abs{y}^{c_0}\right)}.
\]
For example, one can take $\ee < \min(\frac{p-2}{4},\frac{1}{2})$, $c_0=1$, $\theta = 1+2\ee$ and $2+2\delta=\frac{p}{1+2\ee}$. 
Using the techniques from  
\cite{guivarch_spectral_2008}
one can verify that, under Condition \ref{Mom001}, the Banach space $\mathcal{B}$ and the perturbed operator 
$\mathbf P_t f (x) = \int_{\bb R} f(x') e^{ i t x' } \mathbf P (x,\dd x'),$ 
satisfy Hypotheses $\mathbf{M1-M5}$ in \cite{ion_grama_rate_2014}. 
The hypothesis $\mathbf{M1}$  is verified straightforwardly. In particular the norm of the Dirac measure $\delta_x$ is bounded: 
$\norm{\delta_x}_{\mathcal{B}\to \mathcal{B}} \leq 1+ \abs{x}^\theta$, for each $x \in \bb R.$
We refer to Proposition 4 and Corollary 3 of  \cite{guivarch_spectral_2008} for $\mathbf{M2-M3}$. 
For $\mathbf{M4}$, we have
\[
\mu_{\delta}(x) =\underset{k\geq 1}{\sup} \bb E_x^{1/{2+2\delta}}\left( \abs{X_n}^{2+2\delta} \right) \leq c_{\delta} \left( 1+ \abs{x} \right).
\]
Hypothesis $\mathbf{M5}$ follows from Proposition 1 of \cite{guivarch_spectral_2008} and Lemma \ref{MCM}.

With these considerations, the $C(x)=C_1(1+\mu_\delta (x) + \norm{\delta_x}  )^{2+2\delta} $ in  Theorem 3.3 established in \cite{ion_grama_rate_2014} is less than $c_p (1+ \abs{x})^{p},$ where $C_1$ is a constant.  
Therefore Theorem 3.3 can be reformulated in the case of the stochastic recursion as follows.

\begin{proposition}
\label{majdeA_k}
Assume  Condition \ref{Mom001}. 
For any $p \in (2,\alpha)$, there exists $\ee_0 >0$ such that for any $\ee \in (0,\ee_0]$, $x\in \bb R$ and $n\geq 1$, 
without loss of generality (on an extension of the initial probability space) one can reconstruct the sequence $(S_n)_{n\geq 0}$ with a continuous time Brownian motion $(B_t)_{t\in \bb R_{+} }$, such that  
	\[
	\bb P_x \left( \underset{0 \leq t \leq 1}{\sup} \abs{S_{\pent{tn}}-\sigma B_{tn}} > n^{1/2-\ee} \right) 
	\leq \frac{c_{p,\ee}}{n^\ee} (1+\abs{x} )^p,
\]
where $\sigma$ is given by (\ref{defdesigma}).
\end{proposition}

This proposition plays the crucial role in the proof of Theorem \ref{AsExTi} and Theorem \ref{AsCoRaWa} (cf. Sections \ref{As for Exit Time} and \ref{As for cond Markov walk}).
The following straightforward consequence of  Proposition \ref{majdeA_k} is used in the proof of Lemma \ref{concentnu} in Section \ref{CMWI}.
Set $\mathbf \Phi (t)=\frac{1}{\sqrt{2\pi}}\int_{-\infty}^{t} \e^{-\frac{u^2}{2}}\dd u.$

\begin{corollary}
\label{BerEss}
Assume  Condition \ref{Mom001}. 
For any $p \in (2,\alpha)$, there exists $\ee_0 >0$ such that for any $\ee \in (0, \ee_0]$, $x\in \bb R$ and $n\geq 1$,
	\[
	\underset{u\in \bb R}{\sup} \, \abs{ \bb P_x \left( \frac{S_n}{\sqrt{n}} \leq u \right) - \mathbf \Phi \left( \frac{u}{\sigma} \right) } 
	\leq \frac{c_{p,\ee}}{n^\ee} \left(1+\abs{x}\right)^p.
\]
\end{corollary}

\begin{proof} Let $\ee \in (0,1/2)$ and $A_n = \left\{ \underset{0\leq t \leq 1}{\sup} \, \abs{S_{\pent{tn}}-\sigma B_{tn}} > n^{1/2-\ee} \right\}. $ For any $x \in \bb R$ and any $u \in \bb R,$
\[
\bb P_x \left( \frac{S_n}{\sqrt{n}} \leq u \right) \leq \bb P_x \left( A_n \right) + \bb P_x \left( \frac{\sigma B_n}{\sqrt{n}} \leq u + \frac{1}{n^\ee} \right),
\]
where the last probability does not exceed $\mathbf \Phi (\frac{u}{\sigma}) + c_{\ee}n^{-\ee}.$
Using Proposition \ref{majdeA_k}, we conclude that there exists $\ee_0 > 0$ such that for any $\ee \in (0,\ee_0]$ and $x \in \bb R$,
\[
\bb P_x \left( \frac{S_n}{\sqrt{n}} \leq  u \right) \leq \mathbf \Phi \left(\frac{u}{\sigma} \right) + \frac{c_{p,\ee}}{n^\ee} \left(1+\abs{x}\right)^p.
\]
In the same way we obtain a lower bound and the assertion follows.
\end{proof}

\subsection{Finiteness of the exit times} \label{secExitfinit}

\begin{corollary}
\label{Exitfinit}
Assume Condition \ref{Mom001}. 
For any $x\in \bb R$ and $y \in \bb R$,
	\[
\bb P_x	\left( \tau_y < +\infty  \right) =1 \qquad \text{and} \qquad   \bb P_x \left( T_y < +\infty \right) =1.
\]
\end{corollary}

\begin{proof}  
Let $y>0$ and $\ee \in (0,1/2).$ Set $A_n=\left\{  \sup_{0\leq t \leq 1} \abs{S_{[tn]} - \sigma B_{tn}}    \leq n^{1/2-\ee}  \right\}.$
Using Proposition \ref{majdeA_k}, there exists $\ee_0 > 0$ such that for any $\ee \in (0,\ee_0]$, $x\in \bb R$ and $y>0$,
\begin{align*}
	\bb P_x \left( \tau_y > n \right) &\leq \bb P_x \left( \tau_y > n,  A_n \right) + \bb P_x \left( \overline A_n \right) \\
	&\leq \bb P \left( \tau_{y+n^{1/2-\ee}}^{bm} > n \right) + \frac{c_{p,\ee}}{n^\ee} \left(1+\abs{x}\right)^p.
\end{align*}
Since, by the claim 1 of Corollary \ref{exittimeforB},
$\bb P \left( \tau_{y+n^{1/2-\ee}}^{bm} > n \right) \leq  c\frac{y+n^{1/2-\ee}}{\sqrt{n}} \leq  (1+y)\frac{c}{n^\ee},$
taking the limit as $n\to +\infty$ we conclude that $\bb P_x \left( \tau_y < +\infty \right)=1.$

Let $D_n=\left\{  \max_{1\leq k \leq n} \abs{S_{k} - M_{k}}    \leq n^{1/2-\ee}  \right\} $. 
Obviously
\begin{align*}
	\bb P_x \left( T_y > n \right) &\leq \bb P_x \left( T_y > n,  A_n, D_n \right) + \bb P_x \left( \overline A_n \right) + \bb P_x \left( \overline D_n \right) \\
	&\leq \bb P \left( \tau_{y+2n^{1/2-\ee}}^{bm} > n \right) + \frac{c_{p,\ee}}{n^\ee} \left(1+\abs{x}\right)^p + \bb P_x\left(   \max_{1\leq k \leq n} \abs{ \rho X_{k} }  > n^{1/2-\ee}  \right).
\end{align*}
Using the claim 1 of Corollary \ref{exittimeforB}, the Markov inequality and Lemma \ref{MCM},
for any $\ee \in (0,\ee_0]$, $x\in \bb R$ and $y>0$,
\begin{align*}
	\bb P_x \left( T_y > n \right) &\leq     (1+y)\frac{c}{n^\ee}  + \frac{c_{p,\ee}}{n^\ee} \left(1+\abs{x}\right)^p 
	+ c_p\frac{1+\abs{x}^p}{  n^{\frac{p-2}{2} -p\ee}  }.
\end{align*}
Choosing $\ee$ small enough and taking the limit as $n\to +\infty$ we conclude the second assertion when $y>0$.

When $y \leq 0$, the results follow since the applications $y \mapsto \tau_y$ and  $y \mapsto T_y$ are non-decreasing.
\end{proof}

\bibliographystyle{plain}
\bibliography{biblioarXiv01}

\end{document}